\documentclass[letterpaper,10pt,twocolumn]{IEEEtran}

\usepackage{amsmath,amssymb,amsfonts}
\usepackage{amsthm}
\usepackage{graphicx,color}
\usepackage{enumitem}
\usepackage{cite}
\usepackage{algorithm,algorithmic}
\usepackage{subfigure}
\usepackage{wasysym}
\usepackage[colorlinks = true, linkcolor = blue, citecolor =
blue]{hyperref}

\newtheorem{theorem}{Theorem}[section]
\newtheorem{proposition}[theorem]{Proposition}
\newtheorem{lemma}[theorem]{Lemma}
\newtheorem{corollary}[theorem]{Corollary}
\newtheorem{remark}[theorem]{Remark}

\newcommand{\longthmtitle}[1]{\mbox{}\textit{(#1).}}
\newcommand{\overbar}[1]{\mkern 1.5mu\overline{\mkern-1.5mu#1\mkern-1.5mu}\mkern 1.5mu}

\newcommand{\real}{\ensuremath{\mathbb{R}}}
\newcommand{\realpos}{\ensuremath{\mathbb{R}_{>0}}}
\newcommand{\realnonneg}{\ensuremath{\mathbb{R}_{\ge 0}}}
\newcommand{\integers}{{\mathbb{Z}}}
\newcommand{\integerspos}{{\mathbb{N}}}
\newcommand{\integersnonneg}{{\mathbb{N}}_0}

\newcommand{\intersect}{\ensuremath{\operatorname{\cap}}}
\newcommand{\Enorm}[1]{\|#1\|}

\newcommand{\intrangecc}[2]{[#1, #2]_\integers}
\newcommand{\intrangeoo}[2]{(#1, #2)_\integers}
\newcommand{\intrangeoc}[2]{(#1, #2]_\integers}

\newcommand{\indfun}[1]{ \mathbf{1}_{#1} }

\newcommand{\Bc}{\mathcal{B}}

\newcommand{\Dc}{\mathcal{D}}
\newcommand{\Ec}{\mathcal{E}}
\newcommand{\Fc}{\mathcal{F}}
\newcommand{\Gc}{\mathcal{G}}

\newcommand{\Pc}{\mathcal{P}}

\newcommand{\Tc}{\mathcal{T}}

\newcommand{\drm}{\mathrm{d}}

\newcommand{\expect}[2][]{\mathbb{E}_{#1}\left[ #2 \right]}
\newcommand{\condexpect}[3][]{\mathbb{E}_{#1}\left[ #2 \ \vert \ #3 \right]}
\newcommand{\abar}{\bar{a}}
\newcommand{\Abar}{\bar{A}}
\newcommand{\Mbar}{\overbar{M}}
\newcommand{\xhat}{\hat{x}}
\newcommand{\pf}{h}
\newcommand{\pfb}{\bar{h}}
\newcommand{\gD}[1]{g_D(#1)}
\newcommand{\gDD}{g_D}
\newcommand{\LRT}[1]{R_{#1}}    
\newcommand{\ETP}{\Tc_{\Ec}}
\newcommand{\ETPvec}{\bar{\Tc}_{\Ec}}
\newcommand{\st}{s_{*}}
\newcommand{\sst}{s_{**}}
\newcommand{\Ft}{F_{*}}
\newcommand{\Fst}{F_{**}}
\newcommand{\G}[2]{G_{#1}^{#2}}
\newcommand{\Gp}[2]{J_{#1}^{#2}}
\newcommand{\qp}[2]{w_{#1}^{#2}}
\newcommand{\norm}[1]{\| #1 \|}
\newcommand{\tr}[1]{\text{tr}(#1)}

\newcommand{\oprocendsymbol}{\hbox{$\bullet$}}
\newcommand{\oprocend}{\relax\ifmmode\else\unskip\hfill\fi\oprocendsymbol}

\newenvironment{myquote}[1]%
  {\list{}{\leftmargin=#1\rightmargin=#1}\item[]}%
  {\endlist}

\newcommand{\myclearpage}{\clearpage}
\renewcommand{\myclearpage}{}

\allowdisplaybreaks
\graphicspath{{figures/}}

\begin{document}

\title{Event-triggered second-moment stabilization of linear systems
  under packet drops\thanks{A preliminary version of this paper
    appeared at the Allerton Conference on Communications, Control and
    Computing as~\cite{PT-MF-JC:16-allerton}.}}

\author{Pavankumar Tallapragada \qquad Massimo Franceschetti \qquad
  Jorge Cort{\'e}s \thanks{P. Tallapragada is with the Department of
    Electrical Engineering, Indian Institute of Science,
    M. Franceschetti is with the Department of Electrical and Computer
    Engineering, University of California, San Diego and J.
    Cort{\'e}s is with the Department of Mechanical and Aerospace
    Engineering University of California, San Diego {\tt\small
      pavant@ee.iisc.ernet.in, \{massimo,cortes\}@ucsd.edu}}%
}

\maketitle

\begin{abstract}
  This paper deals with the stabilization of linear systems with
  process noise under packet drops between the sensor and the
  controller. Our aim is to ensure exponential convergence of the
  second moment of the plant state to a given bound in finite
  time. Motivated by considerations about the efficient use of the
  available resources, we adopt an event-triggering approach to design
  the transmission policy. In our design, the sensor's decision to
  transmit or not the state to the controller is based on an online
  evaluation of the future satisfaction of the control objective. The
  resulting event-triggering policy is hence specifically tailored to
  the control objective. We formally establish that the proposed
  event-triggering policy meets the desired objective and quantify its
  efficiency by providing an upper bound on the fraction of expected
  number of transmissions in an infinite time interval.  Simulations
  for scalar and vector systems illustrate the results.
\end{abstract}

\section{Introduction}

One of the fundamental abstractions of cyber-physical systems is the
idea of networked control systems, the main characteristic feature of
which is that feedback signals are communicated over a communication
channel or network. As a result, control must be performed under
communication constraints such as quantization, unreliability, and
latency. These limitations make it necessary to design control systems
that tune the use of the available resources to the desired level of
task performance. With this goal in mind, this paper explores the
design of event-triggered transmission policies for second-moment
stabilization of linear plants under packet~drops.

\subsubsection*{Literature review}

The increasing deployment of cyberphysical systems has brought to the
forefront the need for systematic design methodologies that integrate
control, communication, and computation instead of independently
designing these components and integrating them in an adhoc manner,
see e.g.~\cite{KDK-PRK:12,JS-XK-GK-NK-PA-VG-BG-JB-SW:12}. Among this
growing body of literature, the contents of this paper are
particularly related to works that deal with feedback control under
communication constraints,
see~\cite{ST-SM:04,GNN-FF-SZ-RJE:07,SY-TB:13} and references therein,
and specifically packet drops or erasure channels, see
e.g.,~\cite{LS-BS-MF-KP-SSS:06,VG-NCM:10,VG:14-sv}.  In the past
decade, opportunistic state-triggered control
methods~\cite{PT:07,XW-MDL:11,WPMHH-KHJ-PT:12}, have gained popularity
for designing transmission policies for networked control systems that
seek to efficiently use the communication resources. The main idea
behind this approach is to design state-dependent triggering criteria
that opportunistically specify when certain actions (updating the
actuation signal, sampling data, or communicating information) must be
executed. More generally, the triggering criteria may also depend on
the desired control objective, and the available information about the
state, communication channel, and other constraints. In the context of
the communication service, the emphasis has largely been on minimizing
the number of transmissions rather than the quantized data, often
ignoring the limits imposed by channel characteristics, with some
notable exceptions,
see~\cite{DL-JL:10,LL-XW-MDL:12a,YS-XW:14,EG-PJA:13,JP-JPH-DL:14} and
references therein. In our previous
work~\cite{PT-JC:16-tac,PT-MF-JC:15-auto}, we have also sought to
address these limitations for deterministic models of the behavior of
the communication channel. Although today there exists a large body of
work on opportunistic state-triggered control, the application of
these ideas in the stochastic setting is still relatively
limited. This is despite the fact that one of the first works on
event-triggered control~\cite{KJA-BMB:02} was in this setting.
Event-triggering methods in the stochastic setting have almost
exclusively been utilized in finite or infinite horizon optimal
control problems with fixed threshold-based triggering.  The
works~\cite{TH-EJ-AC:08,XM-TC:12,BD-VG-DEQ-MJ:15} also incorporate
transmission costs in the cost function and analyze the optimal
transmission costs. On the other hand, \cite{MR-KHJ:09,RB-FA:12}
analyze the transmission rates. In addition,
\cite{BD-VG-DEQ-MJ:15,MR-KHJ:09,RB-FA:12,MHM-DT-AM-SH:14} also
consider packet drops. The work~\cite{AM-SH:13} shows optimality of
certainty equivalence in event-triggered control for certain finite
horizon problems. In contrast to starting with an event-triggered
control policy, the work~\cite{OCI-TB:06} formulates an optimal
control problem over a finite horizon with the constraint that at most
a smaller number of transmissions may occur, and the optimal control
policy turns out to be event-triggered.  Finally, we should remark
that stochastic stability, in the sense of moment stability, with
event-triggered control has received much less attention. The
work~\cite{RPA-DM-DVD:15} follows~\cite{PT:07} to study self-triggered
sampling for second-moment stability of state-feedback controlled
stochastic differential equations. The work~\cite{DEQ-VG-WM-SY:14}
proposes a fixed threshold-based event-triggered anytime control
policy under packet drops. It assumes that the controller has
knowledge of the transmission times, including when a packet is
dropped, and the policy guarantees second-moment stability with
exponential convergence to a finite bound
asymptotically. Both~\cite{RPA-DM-DVD:15,DEQ-VG-WM-SY:14} are
applicable to multi-dimensional nonlinear systems.

\subsubsection*{Statement of contributions}

We formulate the problem of second-moment stabilization of scalar
linear systems subject to process noise and independent identically
distributed packet drops in the communication channel. Our goal is to
design a policy to prescribe transmissions from the sensor to the
controller that ensures exponential convergence in finite time of the
second moment of the plant state to an ultimate bound. Our first
contribution is the design of an event-triggered transmission policy
in which the decision to transmit or not is determined by a
state-based criterion that uses the available information. The
synthesis of our policy is based on a two-step design
procedure. First, we consider a nominal quasi-time-triggered policy
where no transmission occurs for a given number of timesteps, and then
transmissions occur on every time step thereafter.  Second, we define
the event-trigger policy by evaluating the expectation of the system
performance at the next reception time given the current information
under the nominal policy, and prescribe a transmission if this
expectation fails to meet the objective.  This approach results in a
transmission policy more complex than a threshold-based triggering,
but since it is driven by the control objective results in fewer
transmissions. Our second contribution is the rigorous
characterization of the system evolution, first under the proposed
family of nominal transmission policies and second, building on this
analysis, under the proposed event-triggered transmission policy.
This helps us identify sufficient conditions on the ultimate bound,
the system parameters, and the communication channel that guarantee
that the event-triggered policy indeed meets the control objective.
Our third contribution compares the efficiency of the proposed design
with respect to time-triggered policies and provides an upper bound on
the fraction of the expected number of transmissions over an infinite
time horizon. Our fourth and last contribution is the extension of our
exponential convergence guarantees to the vector case and a discussion
of the design and analysis challenges in extending the
characterization of efficiency.  Various simulations illustrate our
results. We omit the proofs that appeared in the conference
version~\cite{PT-MF-JC:16-allerton} of this work and instead refer the
interested reader there.



\subsubsection*{Notation}

We let $\real$, $\realnonneg$, $\integers$, $\integerspos$,
$\integersnonneg$ denote the set of real, non-negative real numbers,
integers, positive integers and non-negative integers respectively. We
use the notation $\intrangecc{a}{b}$ and $\intrangeoo{a}{b}$ to denote
$[a,b] \intersect \integers$ and $(a,b) \intersect \integers$,
respectively. We use similar notation for half-open/half-closed
intervals. For a matrix $A$, we let $\tr{A}$ denote the trace of the
matrix. Given a set $A$, we denote its indicator function by
$\indfun{A}$, i.e., $\indfun{A}(x) = 0$ if $x \notin A$ and
$\indfun{A}(x) = 1$ if $x \in A$. We use `w.p.' as a shorthand for
`with probability'. We denote the expectation given a transmission
policy $\Pc$ as $\expect[\Pc]{.}$.  Let $(\Omega, \Fc, P)$ be a
probability space and $\Gc_1 \subset \Gc_2 \subset \Fc$ be two
sub-sigma fields of $\Fc$. Then, the \emph{tower property} of
conditional expectation is
\begin{equation*}
  \condexpect{ \condexpect{ X }{ \Gc_2 } }{ \Gc_1 } = \condexpect{ X
  }{ \Gc_1 } = \condexpect{ \condexpect{ X }{ \Gc_1 } }{ \Gc_2 } .
\end{equation*}

\myclearpage
\section{Problem statement}\label{sec:prob-stat}

This section describes the model for the plant dynamics and the
assumptions on the sensor, actuator, and the communication channel
between them. Given this setup, we then specify the objective for the
control design.

\subsubsection*{Plant, sensor, and actuator}

Consider a scalar discrete-time linear time-invariant system evolving
according to
\begin{align}\label{eqn:plant_dyn}
  x_{k+1} = a x_k + u_k + v_k,
\end{align}
for $k \in \integersnonneg$. Here $x \in \real$ denotes the state of
the plant, $a \in \real$ defines the system internal dynamics, $u \in
\real$ is the control input, and $v$ is a zero-mean independent and
identically distributed process noise with covariance $M>0$,
uncorrelated with the system state.

A sensor measures the plant state $x_k$ at time $k$. The sensor, being
not co-located with the controller, communicates with it over an
unreliable communication channel. The sensor maintains an estimate of
the plant state given the `history' (defined precisely below) up to
time $k-1$. During the time between two \emph{successful}
communications, the controller itself estimates the plant state. We
let $\xhat_k^+$ be the controller's estimate of the plant state $x_k$
given the past history of transmissions and receptions including those
at time $k$, if any. This results in a control action given by
$u_k = L \xhat_k^+$. We assume that the sensor can independently
compute $\xhat_k^+$ at the next time step $k+1$ for each
$k \in \integersnonneg$ (this is possible with acknowledgments from
the controller to the sensor on successful reception times). We denote
the \emph{sensor estimation error} and \emph{controller estimation
  error} as $e_k \triangleq x_k - \hat{x}_k$ and
$e_k^+ \triangleq x_k - \hat{x}_k^+$, which are known to the sensor at
all times, but not to the controller.

\subsubsection*{Communication channel}

The sensor can transmit the plant state to the controller with
infinite precision and instantaneously at time steps of its choosing,
but packets might be lost.  We define the \emph{transmission process}
$\{t_k\}_{k \in \integersnonneg}$ as
\begin{equation}
  \label{eq:transmit-proc}
  t_k \triangleq
  \begin{cases}
    1, \quad \text{if a packet is transmitted at } k,
    \\
    0, \quad \text{if no packet is transmitted at } k .
  \end{cases}
\end{equation}
The way in which this process occurs is determined by a transmission
policy $\Tc$, to be specified by the designer.  Similarly, we define a
\emph{reception process} $\{r_k\}_{k \in \integersnonneg}$, with $r_k$
being $1$ or $0$ depending on whether a packet is received or not
at~$k$.  The transmission and reception processes may differ due to
Bernoulli-distributed packet drops.  Formally, if $p \in (0,1]$
denotes the probability of successful transmission, the reception
process is 
\begin{equation}
  \label{eq:recep-proc}
  r_k \triangleq
  \begin{cases}
    1, \quad \text{w.p. } p \text{ if } t_k = 1 ,
    \\
    0, \quad \text{if } t_k = 0 \text{ or w.p. } (1-p) \text{ if } t_k
    = 1 .
  \end{cases}
\end{equation}
We denote the \emph{latest reception time before $k$} and \emph{latest
  reception time up to $k$} by $\LRT{k}$ and $\LRT{k}^+$,
resp. Formally,
  \begin{subequations}\label{eq:Sjk-Spjk}
    \begin{align}
      \LRT{k} & \triangleq \max \{ i < k : r_i = 1 \} ,
      \\
      \LRT{k}^+ & \triangleq \max \{ i \leq k : r_i = 1 \} .
    \end{align}  
  \end{subequations}
  Both times coincide if $r_k=0$. The need for separate notions would
  become clearer later: the notion of $\LRT{k}$ plays a role in the
  design of the triggering rule, while the notion of $\LRT{k}^+$ is
  useful in the analysis of the system evolution.  We denote the
  sequence of all (successful) reception times as
  $\{S_j\}_{j \in \integersnonneg}$, i.e.,
\begin{equation}
  \label{eq:Sj}
  S_0 = 0, \quad S_{j+1} \triangleq \min \{ k > S_j : r_k = 1 \} ,
\end{equation}
where we have assumed, without loss of generality, that $S_0 = 0$ and
hence also $r_0 = 1$. Thus, $S_j$ is the $j^{\text{th}}$ reception
time.

\subsubsection*{System evolution}

Given the sensor-controller communication model specified above, we
describe the system evolution and the controller's estimate,
respectively, as
\begin{subequations}\label{eq:sys-evolve}
  \begin{align}
    x_{k+1} &= a x_k + L \xhat_k^+ + v_k = \abar x_k - L e_k^+ +
    v_k, \label{eq:x-evolve}
    \\
    \xhat_{k+1} &= \abar \xhat_k^+ ,
  \end{align}
  where $\abar = a+L$ and
  \begin{equation}
    \label{eq:xhat_p}
    \xhat_k^+ \triangleq
    \begin{cases}
      \xhat_k \quad &\text{if } r_k = 0,
      \\
      x_k, \quad &\text{if } r_k = 1 .
    \end{cases}
  \end{equation}
\end{subequations}
The use of $\xhat^+$ and $e^+$ is motivated by our goal of designing a
state-triggered transmission policy: the decision to transmit at time
$k$ is made by the sensor based on $x_k$ and $\xhat_k$ (or
equivalently $e_k$), while the plant state at $k+1$ depends on whether
a packet is received or not at~$k$, which is captured by $\xhat^+$
and~$e_k^+$.  We denote by
$I_k \triangleq (k, x_k, e_k, \LRT{k}, x_{\LRT{k}})$ the information
available to the sensor at time $k$, based on which it decides whether
to transmit or not. We also let
$I_k^+ \triangleq (k, x_k, e_k^+, \LRT{k}^+, x_{\LRT{k}^+})$ be the
information available at the controller at time $k$, which can also be
independently computed by the sensor at the end of time step $k$ upon
receiving or not receiving an acknowledgment. Note that $I_k^+$
differs from $I_k$ only if $k$ is a reception time, i.e., $r_k = 1$
(equivalently, only if $k = S_j$ for some $j$). The closed-loop system
is not fully defined until a transmission policy~$\Tc$, determining
the transmission process~\eqref{eq:transmit-proc}, is specified. This
specification is guided by the control objective detailed next.

\subsubsection*{Control objective}

Our objective is to ensure the stability of the plant dynamics with a
guaranteed level of performance.  We rely on stochastic stability
because the presence of random disturbances and the unreliable
communication channel make the plant evolution stochastic. Formally,
we seek to synthesize a transmission policy $\Tc$ ensuring
\begin{equation}\label{eq:overall-objective}
  \condexpect[\Tc]{x_k^2}{I_0^+} \leq \max \{ c^{2k} x_0^2 , B \},
  \quad \forall k \in \integerspos ,
\end{equation}
which corresponds to the second moment of the plant state, conditioned
on the initial information, converging at an exponential rate
$c \in (0,1)$ to its ultimate bound~$B\ge 0$.

A possible, purely time-triggered transmission policy to
guarantee~\eqref{eq:overall-objective} would be to transmit at every
time instant.  Such policy would presumably lead to an inefficient use
of the communication channel, since it is oblivious to the plant state
in making decisions about transmissions.  Instead, we seek to design
an event-triggered transmission policy $\Tc$, i.e., an online policy
in which the decision to transmit or not is determined by a
state-based criterion that uses the available information.

\subsubsection*{Standing assumptions}

We assume the drift constant $a$ is such that $|a| > 1$, so that
control is necessary. We also assume $\abar^2 < c^2 < 1$, so that the
performance function is always non-positive under zero noise and no
packet drops. Finally, we assume $a^2 (1-p) < 1$.  This latter
condition is necessary for second-moment stabilizability under
Bernoulli packet drops, see e.g.~\cite{PM-MF-SD-GNN:09, MF-PM:14}. In
our discussion, the condition is necessary for the convergence of
certain infinite series (we come back to this point in
Remark~\ref{rem:necc-condition-standing}).  For the reader's
reference, we present in the appendix a list of the symbols most
frequently used along the paper.

\myclearpage
\section{Event-triggered transmission policy}\label{sec:design}

This section provides an alternative control objective and shows that
its satisfaction implies the original one defined in
Section~\ref{sec:prob-stat} is also satisfied. This reformulated
objective serves then as the basis for our design of the
event-triggered transmission policy.

\subsection{Online control
  objective}\label{sec:alternative-control-objective}

The control objective stated in~\eqref{eq:overall-objective}
prescribes, given the initial condition, a property on the whole
system trajectory in a priori fashion. This `open-loop' nature makes
it challenging to address the design of the transmission policy. To
tackle this, we describe here an alternative control objective which
prescribes a property on the system trajectory in an online fashion,
making it more handleable for design, and whose satisfaction implies
the original objective is also met. To this end, consider the
\emph{performance function},
\begin{equation}
  \label{eq:pf}
  \pf_k = x_k^2 - \max \{ c^{2(k-\LRT{k})} x_{\LRT{k}}^2 , B \} ,
\end{equation}
which has the interpretation of capturing the desired performance at
time $k$ with respect to the state at the latest reception time
before~$k$. Given this interpretation, consider the alternative
control objective that consists of ensuring that 
\begin{align}\label{eq:objective}
  \condexpect[\Tc]{h_k}{I_{\LRT{k}}^+} \leq 0 , \quad \forall k \in
  \integerspos .
\end{align}
The next result shows that the satisfaction of~\eqref{eq:objective}
ensures that the original control
objective~\eqref{eq:overall-objective} is also met. The proof relies
on the use of induction and can be found
in~\cite{PT-MF-JC:16-allerton}.

\begin{lemma}\longthmtitle{The online control objective is stronger
    than the original control
    objective~\cite{PT-MF-JC:16-allerton}} \label{lem:objectives}
  If a transmission policy $\Tc$ ensures the online
  objective~\eqref{eq:objective}, then it also guarantees the control
  objective~\eqref{eq:overall-objective}.
\end{lemma}

Given this result, our strategy for control design is to satisfy the
stronger but easier to handle online control
objective~\eqref{eq:objective} rather than working directly with the
original objective~\eqref{eq:overall-objective}.

\subsection{Two-step design strategy: nominal and event-triggered
  transmission policies}\label{sec:two-step-design}

In this section, we introduce our event-triggered design strategy to
meet the control objective. Before giving a full description, we first
detail the design principle we have adopted to approach the
problem. Later, we discuss how our two-step design strategy
corresponds to this design principle.
\begin{myquote}{2ex}
  [\emph{Design principle:}] The fundamental principle of
  event-triggered control is to assess if it is necessary to transmit
  at the current time given the control objective and the available
  information about the system and its state (for example, for
  deterministic discrete-time systems with a perfect channel, a
  transmission may be triggered at time~$k$ only if $h_{k+1}$ would be
  greater than $0$ in the absence of a transmission). If the channel
  is not perfect, then its properties must also be taken into
  consideration when deciding whether to transmit or not (for example,
  if the channel induces time delays bounded by $\gamma$, then
  $h_{k+\gamma}$ must be checked in the absence of a transmission at
  time~$k$).  In order to implement this same basic principle for the
  problem at hand, one needs to address the challenges presented by
  the Bernoulli packet drops and the goal of stochastic stability with
  a strict convergence rate requirement (as specified
  in~\eqref{eq:objective}).  A key observation in this regard is the
  fact that \emph{it is not possible to assess the necessity of
    transmission at a given time~$k$ independently of future actions},
  as the occurrence of the next (random) reception time is determined
  by not only the current action but also the future actions. This
  motivates our two-step design strategy. We assess the necessity of
  transmission 
  using a nominal transmission policy in which there is no
  transmission at the current time~$k$.  Our actual transmission
  policy at that time is then based on the expected performance under
  this nominal transmission policy: if the nominal transmission policy
  deems it `not necessary' to transmit on time~$k$, meaning that the
  performance objective is expected to be met under it, then indeed we
  do not transmit on time~$k$.
\end{myquote}

We next describe our design of the event-triggered transmission
policy. The key idea is the belief that, in the absence of reception
of packets, the likelihood of violating the performance criterion must
increase with time. We refer to this as the \emph{monotonicity
  property}. Therefore, we design a transmission policy that overtly
seeks to satisfy the performance criterion~\eqref{eq:objective} only
at the next (random) reception time in order to guarantee that the
performance objective is not violated at any time step. Later, our
analysis will show that the monotonicity property above does indeed
hold.

We seek to design an event-triggered policy $\Tc$ ensuring
\begin{align*}
  \condexpect[\Tc]{ h_{S_{j + 1}} }{ I_{S_j}^+ } \leq 0 , \quad \text{
    for each } j \in \integersnonneg.
\end{align*}
In general, computing the conditional expectation
for an arbitrary event-triggered transmission policy $\Tc$ is
challenging. This is because the evolution of the system state between
consecutive reception times depends on the transmission instants,
which are in turn determined online by the triggering function of the
state and the specific realizations of the noise and the packet
drops. Therefore, we take the two-step strategy described above:
first, we consider a family of nominal quasi-time-triggered
transmission policies $\Tc_k^D$, for which we can compute
$\condexpect[\Tc_k^D]{ h_{\LRT{k + 1}} }{ I_k }$; then, we use this
expectation under the nominal policy to design the event-triggered
policy.

We start by defining a family of \emph{nominal transmission policies} indexed
by $k \in \integersnonneg$ as
\begin{equation}\label{eq:TkD}
  \Tc_k^D : t_i =
  \begin{cases}
    0, \quad i \in \{k, \ldots, k+D-1\},
    \\
    1, \quad i \geq k+D ,
  \end{cases}
\end{equation}
where $D \geq 1$.  Under this nominal policy, no transmissions occur
for the first $D$ time steps from $k$ to $k+D-1$, and transmissions
occur on every time step thereafter ($D$ is therefore the length of
the interval from time $k$ during which no transmissions occur).  With
the nominal policy, we associate the following
\emph{look-ahead} criterion,
\begin{align}\label{eq:GkD}
  \G{k}{D} &\triangleq \condexpect[\Tc_k^D]{ h_{ \LRT{k+1} } }{ I_k }
  \\
  &= \sum_{s=D}^\infty \condexpect{ h_{ \LRT{k+1} } }{ I_k, \LRT{k+1}
    = k + s } (1-p)^{s - D} p , \notag
\end{align}
which is the conditional expectation of the performance function at
the next reception time, given the information at $k$ under the
transmission policy $\Tc_k^D$. 
This interpretation gives rise to the central idea behind our
\emph{event-triggered transmission policy}: if the criterion is
positive (i.e., the performance objective is expected to be violated
at the next reception time if no transmission occurs for $D$
timesteps, and forever after), then we need to start transmitting
earlier to try to revert the situation before it is too
late. Formally, the event-triggered policy $\ETP$, given the last
successful reception time $R_k = S_j$, is
\begin{subequations}\label{eq:ET-design}
  \begin{equation}
    \label{eq:ETP}
    \ETP : t_k =
    \begin{cases}
      0, \quad \text{if } k \in \{ R_k + 1, \ldots, F_k - 1 \}
      \\
      1, \quad \text{if } k \in \{ F_k, \ldots, S_{j+1} \} ,
    \end{cases}
  \end{equation}
  where
  \begin{equation}\label{eq:Fk}
    F_k \triangleq \min \{ \ell > R_k : \G{\ell}{D} \geq 0 \} .
  \end{equation}
\end{subequations}
Thus, under the proposed policy, the sensor transmits on each time
step starting at $F_k$ (the first time after $R_k = S_j$ when the
look-ahead criterion is positive) until a successful reception occurs
at $S_{j+1}$, for each $j \in \integersnonneg$. The complete
transmission policy is then obtained recursively. In the course of the
paper, we analyze the system under the transmission
policy~\eqref{eq:ET-design}, with respect to an arbitrary reception
time $S_j$. Thus, it is convenient to also introduce the notation
  \begin{equation}\label{eq:Tj}
    T_j \triangleq \min \{ \ell > S_j : \G{\ell}{D} \geq 0 \} ,
  \end{equation}
  which is the first time after $S_j$ when a transmission occurs.

  \begin{remark}\longthmtitle{Interpretation of the parameter
      $D$} \label{rem:role-D}
    {\rm The interpretation of the role of the parameter $D$ depends
      on the context. In the nominal policy $\Tc_k^D$, $D$ has the
      role of \emph{idle duration} from $k$ during which no
      transmissions occur. In the actual event-triggered transmission
      policy~\eqref{eq:ET-design}, $D$ has the role of
      \emph{look-ahead horizon}. Specifically, given the information
      available to the sensor at time $k$, the sign of the look-ahead
      function $\G{k}{D}$ answers the question of whether the sensor
      could afford not to transmit for the next $D$ time steps and
      still meet the control objective. If at a time $k$, $G_k^D < 0$,
      then the sensor can afford not to transmit on time steps
      $\{k, \ldots, k+D-1\}$, as there exists a transmission sequence
      in future, given by the nominal policy, that would satisfy the
      control objective. Thus, at a particular time $k$ when
      $G_k^D < 0$, $D$ may be interpreted as \emph{a lower bound on
        the time-to-go for a required transmission}.  Hence,
      intuitively we can see that, in the actual transmission
      policy~(12), a larger value of $D$ makes the policy more
      conservative, because it requires a longer guaranteed
      no-transmission horizon.  }  \oprocend
  \end{remark}

  \begin{remark}\longthmtitle{Special case of deterministic
      channel} \label{rem:nodrop}
    {\rm It is interesting to look at the transmission
      policy~\eqref{eq:ET-design} in the special case of a
      deterministic channel, i.e., no packet drops ($p = 1$). Observe
      from~\eqref{eq:GkD} that in this case,
      $\G{k}{D} = \condexpect{ h_{ k+D } }{ I_k }$. If additionally
      there were no process noise, then this further simplifies to
      $\G{k}{D} = h_{ k+D }$. Then, the policy~\eqref{eq:ET-design}
      reduces to
      \begin{equation*}
        t_k =
        \begin{cases}
          1, \quad \text{if } \G{k}{D} \geq 0
          \\
          0, \quad \text{if } \G{k}{D} < 0 ,
        \end{cases}
      \end{equation*}
      which is a commonly used event-triggering policy for control
      over deterministic channels, see e.g.,~\cite{PT-JC:16-tac}.
      Thus, the proposed policy~\eqref{eq:ET-design} is a natural
      generalization of the basic principle of event-triggering to
      control over channels with probabilistic packet drops.}
    \oprocend
  \end{remark}


\section{Analysis of the system evolution under the nominal
  policy}\label{sec:nominal-analysis}

Here, we characterize the evolution of the system when operating under
the nominal transmission policy. This characterization is key later to
help us provide performance guarantees of the event-triggered
transmission policy.  

\subsection{Performance evaluation functions and their properties}

The following result provides a useful closed-form expression of the
look-ahead criterion $\G{k}{D}$ as a function of $I_k$. Its proof
appears in~\cite{PT-MF-JC:16-allerton}.

\begin{lemma}\longthmtitle{Closed-form expression for the look-ahead
    function~\cite{PT-MF-JC:16-allerton}} \label{lem:closed-form-G}
  The look-ahead function 
  is well defined and takes the form
  \begin{align*}
    \G{k}{D} & = p \Big[ \gD{\abar^2} x_k^2 + 2 \big( \gD{a \abar} -
    \gD{\abar^2} \big) x_k e_k \notag
    \\
    & \quad + \big( \gD{a^2} - 2 \gD{a \abar} + \gD{\abar^2} \big)
    e_k^2 \notag
    \\
    & \quad + \Mbar \left( \gD{a^2} - \frac{1}{p} \right) - \gD{c^2}
    z_k \notag
    \\
    & \quad - \Big( \frac{B}{p} - c^{2q_k^D} \gD{c^2} z_k \Big)
    (1-p)^{q_k^D} \Big] ,
  \end{align*}
  where
  \begin{align}
    \mkern-20mu \gD{b} & \triangleq \frac{ b^D }{ 1 - b(1-p) }, \ \
    \Mbar \triangleq \frac{ M }{ a^2 - 1 }, \ \ z_k \triangleq
    c^{2(k-\LRT{k})} x_{\LRT{k}}^2 , \notag
    \\
    q_k^D & \triangleq \max \left\{ 0, \left\lceil \frac{ \log \left(
            \frac{ x_{\LRT{k}}^2 }{ B } \right) }{ \log(1/c^2) }
      \right\rceil - ( k - \LRT{k} ) - D \right\} . \label{eq:qkD}
  \end{align}
\end{lemma}

The function $\G{k}{D}$ helps determine whether or not to transmit at
time $k$. However, to analyze the evolution of the performance
function $\pf_k$ between successive reception times $S_j$ and
$S_{j+1}$, we introduce the \emph{performance-evaluation} function,
\begin{align}\label{eq:Gp-kD}
  \Gp{k}{D} &\triangleq \condexpect[\Tc_k^D]{ h_{ \LRT{k+1}^+ } }{
    I_k^+ }
  \\
  &= \sum_{s=D}^\infty \condexpect{ h_{ \LRT{k+1}^+ } }{ I_k^+,
    \LRT{k+1} = k + s } (1-p)^{s - D} p . \notag
\end{align}
Note the similarity with the definition of $\G{k}{D}$ (with the
exception that $ \Gp{k}{D}$ is conditioned upon the
information~$I_k^+$).  Observe that $\Gp{k}{D} \neq \G{k}{D}$ only if
$k = S_j$ for some $j$. Hence we focus on $\Gp{S_j}{D}$ for
$j \in \integersnonneg$,
\begin{equation}\label{eq:Gp-SjD}
  \Gp{S_j}{D} = \sum_{s=D}^\infty H(s, x_{S_j}^2) (1-p)^{s - D} p ,
\end{equation}
where
\begin{equation}
  \label{eq:H-def}
  H(s, x_{S_j}^2) \triangleq \condexpect{ h_{ S_j + s } }{ I_{S_j}^+, S_j
    + s \leq S_{j+1} } ,
\end{equation}
which we call the \emph{open-loop performance evolution} 
function. This function describes the evolution of the expected value
of the performance function in open loop, during the inter-reception
times, conditioned upon $I_{S_j}^+$, the information available at the
last reception time upon reception.

\begin{remark}\longthmtitle{Necessary condition for second-moment
    stability} \label{rem:necc-condition-standing} {\rm The condition
    $a^2(1-p) < 1$, which we assumed in the standing assumption in
    Section~\ref{sec:prob-stat}, is necessary for the convergence of
    the series~\eqref{eq:GkD} and~\eqref{eq:Gp-SjD}, which define the
    look-ahead criterion and performance-evaluation function,
    respectively. This can be seen from the proofs of
    Lemma~\ref{lem:closed-form-G} and Lemma~\ref{lem:closed-form-Gp},
    in~\cite{PT-MF-JC:16-allerton}. The necessity of the condition
    $a^2(1-p) < 1$ for second-moment stability can also be derived
    from the information-theoretic or data-rate arguments employed
    in~\cite{PM-MF-SD-GNN:09, MF-PM:14}.} \oprocend
\end{remark}

The next result gives closed-form expressions for the
performance-evaluation function~$\Gp{S_j}{D}$ 
and the open-loop performance evaluation function~$H$. The proof
appears in~\cite{PT-MF-JC:16-allerton}.

\begin{lemma}\longthmtitle{Closed-form expressions for the
    performance-evaluation  and the open-loop performance evaluation
    functions~\cite{PT-MF-JC:16-allerton}} \label{lem:closed-form-Gp}
  The performance-evaluation function 
  is well defined and takes the form
  \begin{align*}
    \Gp{S_j}{D} & = p \Big[ \gD{\abar^2} x_{S_j}^2 + \Mbar \left(
      \gD{a^2} - \frac{1}{p} \right) - \gD{c^2} x_{S_j}^2 \notag
    \\
    & \quad - \Big( \frac{B}{p} - c^{2 \qp{k}{D} } \gD{c^2} x_{S_j}^2
    \Big) (1-p)^{\qp{S_j}{D}} \Big] , 
  \end{align*}
  where $\gDD$ is defined in~\eqref{eq:qkD} and
  \begin{align}
    \label{eq:qp-SjD}
    \qp{S_j}{D} & \triangleq \max \left\{ 0, \left\lceil \frac{ \log
          \left( \frac{ x_{S_j}^2 }{ B } \right) }{ \log(1/c^2) }
      \right\rceil - D \right\} .
  \end{align}
  The open-loop performance evaluation function 
  takes the form
    \begin{align}\label{eq:H-form}
      H(s, y) = \abar^{2s} y + \Mbar ( a^{2s} - 1 ) - \max \{ c^{2s}
      y, B \} .
    \end{align}
\end{lemma}

The next result specifies some useful properties of the look-ahead
$\G{k}{D}$ and the performance-evaluation $\Gp{k}{D}$ functions.  The
proof appears in~\cite{PT-MF-JC:16-allerton}. 

\begin{proposition}\longthmtitle{Properties of the
    look-ahead and performance-evaluation
    functions~\cite{PT-MF-JC:16-allerton}} \label{prop:GkD-prop} 
  For $D \in \integerspos$, under the same hypotheses as in
  Proposition~\ref{prop:H}, the following hold:
  \begin{itemize}
  \item[(a)] Let $\Tc$ be any transmission policy. Then, for any $k
    \in \integersnonneg$,
    \begin{align*}
      &\condexpect[\Tc]{ \G{k+1}{D} }{ I_k, r_k = 0 } = \G{k}{D+1},
      \\
      &\condexpect[\Tc]{ \G{k+1}{D} }{ I_k, r_k = 1 } = \Gp{k}{D+1} .
    \end{align*}
  \item[(b)] For $\Dc \in \integerspos$, define
    \begin{equation}\label{eq:param-constraint}
      \mkern-10mu  \Gc(\Dc) \triangleq \left( g_\Dc( \abar^2 ) -
        g_\Dc(c^2) \right) \frac{ B }{ c^{2\Dc} }  
      + \Mbar \left( g_\Dc(a^2) - \frac{1}{p} \right) .
    \end{equation}
    If $\Gc(\Dc) < 0$, then $\Gp{S_j}{\Dc} < 0$, for any
    $j \in \integersnonneg$.
  \item[(c)] Suppose the hypothesis of (b) is true. Then, for
    $d \in \{ 1, \ldots, \Dc \}$ and for any $j \in \integersnonneg$,
    $\Gp{S_j}{d} \leq \Gp{S_j}{d+1}$.
  \end{itemize}
\end{proposition}

The value of the function $\Gc$ (defined
in~\eqref{eq:param-constraint}) at $D$ has the interpretation of being
a uniform (over the plant state space) upper bound on $\Gp{S_j}{D}$,
the expectation of the open-loop performance function at the next
(random) reception time.
The condition $\Gc(D)<0$ can be interpreted as establishing a lower
bound on the value of $B$, the ultimate bound, as a function of the
system and communication channel parameters.  The next result
establishes a useful property of~$\Gc$ which would be useful in our
forthcoming analysis.

\begin{lemma}\longthmtitle{The function $\Gc$ is strictly
    increasing} \label{cor:Gc}
  Under the hypotheses of Proposition~\ref{prop:H}, the function $\Gc$
  (cf.~\eqref{eq:param-constraint}) is strictly increasing on
  $[0, \infty)$.
\end{lemma}
\begin{proof}
  The derivative of $\Gc$ with respect to $\Dc$ is
  \begin{align*}
    \frac{ \mathrm{d} \Gc }{ \mathrm{d} \Dc}
    &= \Mbar \log(a^2) \frac{ a^{2\Dc} }{ 1 - a^2(1-p) } - B \log
      \left( \frac{ c^2 }{ \abar^2 } \right) \frac{ \abar^{2\Dc} }{ 1
      - \abar^2(1-p) }
    \\
    &< B \log \left( \frac{ c^2 }{ \abar^2 } \right) \left[ \frac{
      a^{2\Dc} }{ 1 - a^2(1-p) } - \frac{ ( \abar / c )^{2\Dc} }{ 1 -
      \abar^2(1-p) } \right] ,
  \end{align*}
  where the inequality follows from the assumption that $B > B^*$
  and~\eqref{eq:Bc}. Then, observe that for $\Dc \geq 0$
  \begin{align*}
    &a^{2\Dc} ( 1 - \abar^2(1-p) ) - ( \abar / c )^{2\Dc} ( 1 - a^2(1-p)
      )
    \\
    &> (a^{2\Dc} - ( \abar / c )^{2\Dc}) ( 1 - a^2(1-p) ) > 0 ,
  \end{align*}
  where the inequalities follow form the fact $\abar^2 < c^2 < a^2$.
  Thus, $\frac{ \mathrm{d} \Gc }{ \mathrm{d} \Dc} > 0$ for
  $\Dc \geq 0$.
\end{proof}

\subsection{Monotonicity of the open-loop performance function}

This section establishes the monotonicity of the open-loop performance
function~$H$, which forms the basis for our main results. Recall from
our discussion in Section~\ref{sec:two-step-design} that this property
refers to the intuition that, in the absence of reception of packets,
the likelihood of violating the performance criterion must increase
with time. This property is captured by the following result.

\begin{proposition}\longthmtitle{Monotonicity of the open-loop
    performance function}\label{prop:H} 
  There exists
    \begin{equation}\label{eq:Bc}
      B^*  > B_c \triangleq \Mbar \frac{ \log(a^2) }{ \log \left(
          \frac{ c^2 }{ \abar^2 } \right) } > 0 
    \end{equation}
    such that if $B > B^*$ then for each $y \in \realnonneg$, the
    function $H(. , y)$ has the property:
  \begin{align}\label{eq:Hprop}
    H( s_1, y ) > 0 \implies H( s_2, y ) > 0, \quad \forall s_2 \geq
    s_1 .
  \end{align}
\end{proposition}

Proposition~\ref{prop:H} states that, given the plant state is $y$ at
any reception time $S_j$, then there is a time $s_0$ such that, in the
absence of receptions, the plant state is expected to satisfy the
performance criterion~\eqref{eq:objective} until $S_j + s_0$ and
violate it on every time step thereafter. 

The proof of Proposition~\ref{prop:H} requires a number of
intermediate results that we detail next.  We start by introducing the
functions $f_1,f_2: \realnonneg^2 \rightarrow \real$,
\begin{subequations}
  \label{eq:f1f2}
  \begin{align}
    f_1(s,y) &\triangleq \abar^{2s} y + \Mbar ( a^{2s} - 1 ) - c^{2s}
    y ,
    \label{eq:f1}
    \\
    f_2(s,y) &\triangleq \abar^{2s} y + \Mbar ( a^{2s} - 1 ) - B
    . \label{eq:f2}
  \end{align}
\end{subequations}
Notice, from~\eqref{eq:H-form}, that $H(s,y) = \min \{ f_1(s,y),
f_2(s,y) \}$.
\begin{figure*}[htb!]
  \centering
  \subfigure[Case-I\label{fig:H-case1}]{\includegraphics[width=0.24\textwidth]{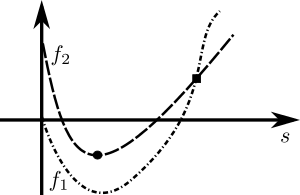}}
  \subfigure[Case-II\label{fig:H-case2}]{\includegraphics[width=0.24\textwidth]{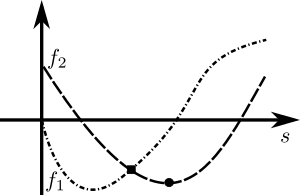}}
  \subfigure[Case-III\label{fig:H-case3}]{\includegraphics[width=0.24\textwidth]{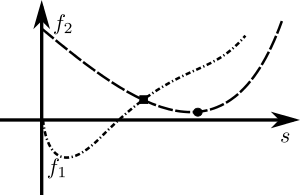}}
  \subfigure[Case-IV\label{fig:H-case4}]{\includegraphics[width=0.24\textwidth]{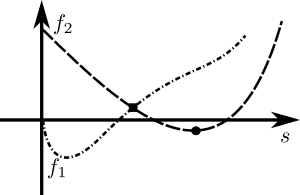}}
  \caption{For $y > B$, there are four possible cases of how
    $f_1(.,y)$ and $f_2(.,y)$ and hence $H(.,y)$ evolve. In the
    figures, $\CIRCLE$ and $\blacksquare$ show the points $(\st, \Ft)$
    and $(\sst, \Fst)$, respectively. In Case-I $s_*(y) < s_{**}(y)$
    and in Cases II-IV $s_*(y) \geq s_{**}(y)$. In addition, in
    Case-II $F_{**}(y) \leq 0$, in Case-III $F_{**}(y) > 0$ and
    $F_*(y) > 0$ and in Case-IV $F_{**}(y) > 0$ and $F_*(y) \leq
    0$.}\label{fig:H-cases}
\end{figure*}
Our proof strategy to establish Proposition~\ref{prop:H} is the
following:
\begin{quote}
  {\it Roadmap:} We first show that $f_2(.,y)$ is strongly convex and
  $f_1(.,y)$ is quasiconvex. Notice that for $y \leq B$, $H(s,y) =
  f_2(s,y)$ for all $s \geq 0$. Thus for $y > B$, we analyze the
  conditions under which one or the other of the functions $f_1(.,y)$
  and $f_2(.,y)$ is the minimum of the two. In this process, we find
  it useful to analyze the relationship between $\st$ and $\sst$, the
  unique point where $f_1(.,y)$ attains its minimum and the unique
  point where $f_1(.,y)$ equals $f_2(.,y)$, respectively. In addition,
  the function values at these points
  \begin{subequations}
    \begin{align}
      \Ft(y) &\triangleq f_2(\st(y), y) \label{eq:Ft}
      \\
      \Fst(y) &\triangleq f_1( \sst(y), y) = f_2( \sst(y), y
      ), \label{eq:Fst}
    \end{align}
  \end{subequations}
  also play an important role. Based on the relationship between
  $(\st, \Ft)$ and $(\sst, \Fst)$, the behavior of the open-loop
  performance function, for $y > B$, can be qualitatively classified
  into four different cases, which are illustrated in
  Figure~\ref{fig:H-cases}. Notice from the plots that $H$ has the
  property~\eqref{eq:Hprop} in all but Case-IV. Thus, the key to the
  proof is in showing that Case-IV does not occur under the hypothesis
  of Proposition~\ref{prop:H}.
\end{quote}

In the sequel, we discuss the various claims alluded to in the above
roadmap.

\begin{lemma}\longthmtitle{Convexity properties
    of~$f_2$}\label{lem:f2-strong-convex} 
  For any fixed $y \in \realnonneg$, the function $f_2( ., y )$ is
  strongly convex.
\end{lemma}
\begin{proof}
  Strong convexity of $f_2$ with respect to $s$ for a fixed $y$
  follows directly by taking the second derivative.
  \begin{equation*}
    \frac{ \partial^2 f_2 }{ \partial s^2 } = \abar^{2s}
    \log^2(\abar^2) y + \Mbar a^{2s} \log^2(a^2) > \Mbar \log^2(a^2) > 0 .
  \end{equation*}
\end{proof}

On the other hand, $f_1(.,y)$ for any fixed $y \in \realnonneg$ is
only quasiconvex in general, as the following result states.

\begin{lemma}\longthmtitle{Convexity properties
    of~$f_1$}\label{lem:f1-quasiconvex} 
  For any fixed $y \in \realnonneg$, the function $f_1(. , y)$ is
  quasiconvex.
\end{lemma}
\begin{proof}
  For any fixed $y \in \realnonneg$, let $g_1(s) \triangleq f_1(s,y)$. Then,
  \begin{equation*}
    g'_1(s) = \abar^{2s} y \log( \abar^2 ) + \Mbar a^{2s} \log( a^2 )
    - c^{2s} y \log( c^2 ) .
  \end{equation*}
  Notice that $g'_1(s)$ has the same sign as
  \begin{align*}
    g_2(s) &\triangleq \frac{ g'_1(s) }{ \abar^{2s} }
    \\
    &= y \log( \abar^2 ) + \Mbar \left( \frac{ a }{ \abar }
    \right)^{2s} \log( a^2 ) - \left( \frac{ c }{ \abar } \right)^{2s}
    y \log( c^2 ) ,
  \end{align*}
  which, by the standing assumptions, is a strictly increasing
  function of $s$. Since $g'_1(s)$ has the same sign as $g_2(s)$, we
  conclude that $g_1 = f_1(.,y)$ is quasiconvex.
\end{proof}

The strong convexity of $f_2(.,y)$ and quasiconvexity of $f_1(.,y)$
are very useful in proving Proposition~\ref{prop:H}. In order to
proceed with the proof, we need to determine the subsets of the domain
where the minimum in the definition of $H$ is achieved by each of the
functions $f_1$ and $f_2$. Thus, we define the function
\begin{equation}
  \label{eq:sst}
  \sst(y) \triangleq \frac{ \log(y) - \log(B) }{ \log(1/c^2) } ,
\end{equation}
that corresponds to the point where $f_1$ and $f_2$ cross each other,
i.e., $H(\sst(y),y) = f_1(\sst(y),y) = f_2(\sst(y),y)$.

\begin{lemma}\longthmtitle{Convexity properties
    of~$H$}\label{lem:H-intervals}
  Given any $y \in \realnonneg$, $H(.,y)$ is quasiconvex on $[0,
  \sst(y)]$ and strongly convex on $[\sst(y), \infty)$.
\end{lemma}
\begin{proof}
  The result follows directly from the definition~\eqref{eq:sst} of
  $\sst(y)$, the facts that
  \begin{align*}
    H(s,y) &= f_1(s,y) < f_2(s,y), \quad \forall s \in [0, \sst(y)) ,
    \\
    H(s,y) &= f_2(s,y) < f_1(s,y), \quad \forall s \in (\sst(y),
    \infty) ,
  \end{align*}
  together with the quasiconvexity of $f_1(.,y)$, cf.
  Lemma~\ref{lem:f1-quasiconvex}, and the strong convexity of
  $f_2(.,y)$, cf. Lemma~\ref{lem:f2-strong-convex}.
\end{proof}


Note that, by itself, this result is not sufficient to ascertain the
convexity properties of $H(.,y)$ over the whole domain
$\realnonneg$. However, if $f_2(.,y)$ is increasing for $s > \sst(y)$,
then Lemma~\ref{lem:H-intervals} would imply that $H(.,y)$ is
quasiconvex on $\realnonneg$, and this together with the fact that
$H(0,y) = 0$, in turn imply that the property~\eqref{eq:Hprop} holds.
Thus, our next objective is to find the values of $y$ for which
$f_2(.,y)$ is increasing for $s > \sst(y)$. To this aim, we find the
minimizer of this function as
\begin{equation}
  \label{eq:st}
  \st(y) \triangleq \log \left( \frac{ y \log \left( \frac{ 1 }{ \abar^2
        } \right) }{ \Mbar \log(a^2) } \right) \frac{ 1 }{ \log \left(
      \frac{ a^2 }{ \abar^2 } \right) } .
\end{equation}
Clearly, if $\st(y) \leq \sst(y)$, then $f_2(.,y)$ would be increasing
for $s > \sst(y)$, as desired.  Therefore, we are interested in the
function
\begin{equation}
  \label{eq:W}
  W(y) \triangleq \st(y) - \sst(y) ,
\end{equation}
and, more specifically, on the sign of $W$ as a function of $y$.

\begin{lemma}\longthmtitle{Monotonic behavior of~$W$} \label{lem:W}
  The function $W$ is monotonically decreasing on $[B, \infty)$ and
  $W(U) = 0$, where $U$ is given by
  \begin{equation}
    \label{eq:U}
    \log(U) \triangleq \frac{ \log \left( \frac{ B \log \left(
            \frac{ 1 }{ \abar^2 } \right) }{ \Mbar \log(a^2) } \right)
      \log \left( \frac{ 1 }{ c^2 } \right) }{ \log \left( \frac{ a^2
          c^2 }{ \abar^2 } \right) } + \log(B) .
  \end{equation}
\end{lemma}
\begin{proof}
  From~\eqref{eq:sst} and \eqref{eq:st}, we see that
  \begin{align*}
    W'(y) &= \left( \frac{ 1 }{ \log \left( \frac{ a^2 }{ \abar^2 }
        \right) } - \frac{ 1 }{ \log \left( \frac{ 1 }{ c^2 } \right)
      } \right) \frac{ 1 }{ y } = \frac{ - \log \left( \frac{ a^2 c^2
        }{ \abar^2 } \right) }{ \log \left( \frac{ a^2 }{ \abar^2 }
      \right) \log \left( \frac{ 1 }{ c^2 } \right) y } < 0 ,
  \end{align*}
  where the last inequality follows from the fact that $a^2 > 1$ and
  $c^2 > \abar^2$ and the fact that $y \in [B, \infty)$. Thus, $W$ is
  monotonically decreasing for $y \in [B, \infty)$. The value of $U$
  can be obtained directly by solving $W(U) = 0$.
\end{proof}


It is clear that if $\Ft(y)$, the minimum value of $f_2(.,y)$,
(see~\eqref{eq:Ft}) is greater than zero then again
property~\eqref{eq:Hprop} is satisfied. Thus, we now note how $\Ft(y)$
evolves with $y$.

\begin{lemma}\longthmtitle{Motonic behavior
    of~$\Ft$}\label{lem:Ft-increasing} 
  The function $\Ft$ is monotonically increasing on $\realpos$.
\end{lemma}
\begin{proof}
  It can be easily verified that
  \begin{equation*}
    \Ft'(y) = \left( \abar^{2 \st(y)} + \frac{ \Mbar a^{2 \st(y)} }{ y
      } \right) \frac{ \log(a^2) }{ \log \left( \frac{ a^2 }{ \abar^2 }
      \right) } > 0 ,
  \end{equation*}
  which proves the result.
\end{proof}


Now, also note that, if $H(\sst(y), y) \leq 0$, then strong convexity
of $H(., y)$ in the interval $[\sst(y), \infty)$ guarantees the
property~\eqref{eq:Hprop}. Thus, we now analyze the evolution of the
function $\Fst(y)$ (see~\eqref{eq:Fst}) with $y$.

\begin{lemma}\longthmtitle{Convexity properties
    of~$\Fst$}\label{lem:Fst-quasiconvex} 
  The function $\Fst$ is quasiconvex on $\realnonneg$.
\end{lemma}
\begin{proof}
  We can easily verify that
  \begin{equation*}
    \Fst'(y) = \left( \abar^{2 \sst(y)} \log \left( \frac{ \abar^2 }{
          c^2 } \right) + \frac{ \Mbar a^{2 \sst(y)} \log(a^2) }{ y
      } \right) \frac{ 1 }{ \log \left( \frac{ 1 }{ c^2 }
      \right) },
  \end{equation*}
  which has the same sign as the function $g(y) \triangleq \Fst'(y)/
  \abar^{2 \sst(y) }$. We can then verify, for all $y>0$,
  \begin{equation*}
    g'(y) = \frac{ \Mbar \left( \frac{ a^2 }{ \abar^2 }
      \right)^{\sst(y)} \log(a^2) \log \left( \frac{ a^2 c^2 }{ \abar^2
        } \right) }{ \log^2 \left( \frac{ 1 }{ c^2 } \right ) y^2 } >
    0.
  \end{equation*}
  Thus, $g$ is strictly increasing, and since $g(y)$ and $\Fst'(y)$
  have the same sign, $\Fst$ is quasiconvex.
\end{proof}

\begin{lemma}\longthmtitle{Choice of $B$} \label{lem:Bstar}
  There exists $B^*>0$ such that $\Fst(U(B^*)) = 0$ and, if $B > B^*$,
  then $\Fst(U(B)) < 0$.
\end{lemma}
\begin{proof}
  We first make explicit the dependence of $U$ on $B$ by
  rewriting~\eqref{eq:U} as
  \begin{equation}
    \label{eq:UB}
    \log( U(B) ) = \frac{ P_1 }{ P_2 } \log(B) + \frac{ P_3 P_4 }{ P_2
    } ,
  \end{equation}
  where
  \begin{align*}
    &P_1 \triangleq \log \left( \frac{ a^2 }{ \abar^2 } \right), \quad
    P_2 \triangleq \log \left( \frac{ a^2 c^2 }{ \abar^2 } \right),
    \\
    &P_3 \triangleq \log \left( \frac{ 1 }{ c^2 } \right), \quad P_4
    \triangleq \log \left( \frac{ \log \left( \frac{ 1 }{ \abar^2 }
        \right) }{ \Mbar \log(a^2) } \right) .
  \end{align*}
  Note that
  \begin{equation*}
    U(B) = e^{ \frac{ P_3 P_4 }{ P_2 } } \cdot B^{ \frac{ P_1 }{ P_2
      } }, \quad \frac{ \drm U }{ \drm B } = \frac{ P_1 U }{ P_2 B } .
  \end{equation*}
  Using the definitions of $P_1$, $P_2$, and $P_3$ in~\eqref{eq:sst},
  we obtain
  \begin{align*}
    \sst( U(B) ) = \frac{ \log(B) }{ P_2 } + \frac{ P_4 }{ P_2 } .
  \end{align*}
  Next, we use this expression to evaluate~\eqref{eq:Fst} and
  establish $\Fst( U(B) ) = Y(B) - \Mbar - B$, where
  \begin{align*}
    Y(B) \triangleq \abar^{2 \sst( U(B) )} U(B) + \Mbar a^{2 \sst(
      U(B) )}.
  \end{align*}
  One can then verify, using the definition of $P_1$ and $P_2$ to
  simplify the expressions, that
  \begin{align*}
    &\frac{ \drm \Fst( U(B) ) }{ \drm B } = \frac{ Y(B) \log(a^2) }{
      P_2 B } - 1 ,
    \\
    &\frac{ \drm^2 \Fst( U(B) ) }{ \drm B^2 } = - \frac{ Y(B)
      \log(a^2) \log \left( \frac{ c^2 }{ \abar^2 } \right) }{ P_2^2
      B^2 } < 0, \ \forall B > 0 .
  \end{align*}
  Thus, the function $\Fst(U(.))$ is a strictly concave function - it
  has at most two zeros and it is positive only between those zeros,
  if they exist. Now, note that for $B_0 = e^{-P_4}$, $\sst(U(B_0)) =
  0$ and hence $\Fst(U(B_0)) = 0$. Therefore, there exists a $B^* \geq
  B_0$ such that $\Fst(U(B^*)) = 0$ and by the strict concavity of
  $\Fst(U(.))$, $\Fst(U(B))$ is strictly decreasing for all $B \geq
  B^*$. This proves the result.
 \end{proof}

 The final arguments of the proof also suggest a method to numerically
 find $B^*$. First, note that $B_0 < B_c$. As a result, if
 $\Fst(U(.))$ is non-increasing at $B_0$ then $B^* = B_c$.  Otherwise,
 the other zero, $B_z$ of $\Fst(U(.))$ can be found by simply marching
 forward in $B$ from $B_0$. Then, $B^* = \max\{ B_c, B_z\}$. Now, all
 the pieces necessary for the proof of Proposition~\ref{prop:H} are
 finally in place.

\begin{proof}[Proof of Proposition~\ref{prop:H}]
  First, notice from the definition~\eqref{eq:sst} of $ \sst$ that if
  $y \leq B$, then $H(s,y) = f_2(s,y)$ for all $s \geq 0$. Then, the
  strong convexity of $f_2(.,y)$,
  cf. Lemma~\ref{lem:f2-strong-convex}, and the fact that $f_2(0,y)
  \leq 0$ for all $y \leq B$ are sufficient to prove
  Proposition~\ref{prop:H}. Therefore, in what follows, we assume that
  $y > B$.

  There are four possible cases that may arise, specified as
  \begin{align*}
    \text{Case-I:} \quad &s_*(y) < s_{**}(y) ,
    \\
    \text{Case-II:} \quad &s_*(y) \geq s_{**}(y) \land F_{**}(y) \leq
    0 ,
    \\
    \text{Case-III:} \quad &s_*(y) \geq s_{**}(y) \land F_{**}(y) > 0
    \land F_*(y) > 0 ,
    \\
    \text{Case-IV:} \quad &s_*(y) \geq s_{**}(y) \land F_{**}(y) > 0
    \land F_*(y) \leq 0 .
  \end{align*}
  Figure~\ref{fig:H-cases} illustrates each of these cases. First,
  note that for $y > B$, $H(0,y) = 0$. Also, recall from
  Lemma~\ref{lem:H-intervals} that $H(.,y)$ is quasiconvex for $s \in
  [0, \sst(y)]$ and thus in this interval, $H$ satisfies the
  property~\eqref{eq:Hprop}. It is only the behavior of $H(s,y)$ for
  $s \in [\sst(y), \infty)$ that is of concern to us.

  Thus in Case-I, since $\st(y) < \sst(y)$ and the strong convexity of
  $f_2(.,y)$, cf. Lemma~\ref{lem:f2-strong-convex}, mean that $H(.,y)$
  is strictly increasing in $[\sst(y), \infty)$, which is sufficient
  to prove property~\eqref{eq:Hprop}. In Case-II, $\Fst(y) \leq 0$ and
  again the strong convexity of $H(.,y)$ in $[\sst(y), \infty)$
  guarantees the result. In Case-III, the fact that $\Ft(y) > 0$
  directly guarantees property~\eqref{eq:Hprop}.

  It is only in Case IV when the property~\eqref{eq:Hprop} would be
  violated. So, now we take into account the assumption that $B >
  B^*$. Notice from~\eqref{eq:W} and Lemma~\ref{lem:W} that in Case
  IV, $y \in [B, U]$. Also notice that $\Fst(B) = 0$ and by
  Lemma~\ref{lem:Bstar} that $\Fst(U) < 0$. Then, the quasiconvexity
  of $\Fst$, cf.  Lemma~\ref{lem:Fst-quasiconvex}, implies that
  $\Fst(y) \leq 0$ for all $y \in [B, U]$, which ensures that Case IV
  does not occur. This completes the proof of
  Proposition~\ref{prop:H}.
\end{proof}

Observe that, in ruling out the occurrence of Case-IV we have also
ruled out the occurrence of Case-III. From Lemma~\ref{lem:Bstar}, we
see that the condition $B > B^*$ is only sufficient and it may seem
that the `good' Case-III has been ruled out inadvertently. However,
note that, by the definitions of $\Ft$ and $\Fst$,
$\Fst(y) \geq \Ft(y)$ for any $y > 0$. Thus, Case-IV is ruled out only
if both $\Fst(y)$ and $\Ft(y)$ are of the same sign for all
$y \in [B, U]$. Therefore, ruling out Case-IV automatically also rules
out Case-III.

\section{Convergence and performance analysis under the
  event-triggered policy}\label{sec:analysis}

In this section, we characterize the convergence and performance
properties of the system evolution operating under the event-triggered
transmission policy~$\ETP$ defined in~\eqref{eq:ET-design}.

\subsection{Convergence guarantees: the control objective is
  achieved}\label{sec:event-analysis}

The following statement is the main result of the paper and shows that
the control objective is achieved by the proposed event-triggered
transmission policy.

\begin{theorem}\longthmtitle{The event-triggered policy meets the control
    objective}\label{thm:main}
  If the ultimate bound satisfies $B > B^*$ and $D \in \integerspos$
  is such that $\Gc(D) < 0$, cf.~\eqref{eq:param-constraint}, then the
  event-triggered policy $\ETP$ guarantees that
  $\condexpect[\ETP]{ \pf_{k} }{ I_{\LRT{k}}^+ } \leq 0$ for all
  $k \in \integerspos$.
\end{theorem}
\begin{proof}
  We structure the proof around the following two claims.

  \emph{Claim (a):} For any $j \in \integerspos$,
  $\condexpect[\ETP]{ \pf_{S_{j+1}} }{ I_{S_j}^+ } \leq 0$ implies
  $\condexpect[\ETP]{ \pf_{k} }{ I_{S_j}^+ } \leq 0$ for all $k \in
  \intrangecc{S_j}{S_{j+1}}$.

  \emph{Claim (b):} For any $j \in \integerspos$, $\condexpect[\ETP]{
    \pf_{S_{j+1}} }{ I_{S_j}^+ } < 0$.


  Note that if both the claims hold, the result automatically follows.
  Therefore, it now suffices to establish claims (a) and (b). Towards
  this aim, first observe that
  \begin{equation*}
    \condexpect[\ETP]{ \pf_k }{ I_{S_j}^+ } =  \condexpect{
      \pf_k }{ I_{S_j}^+ }, \quad \forall k \in
    \intrangecc{S_j}{S_{j+1}} .
  \end{equation*}
  This can be reasoned by noting that a transmission policy only
  affects the sequence of reception times, $\{S_j\}_{j \in
    \integerspos}$, and has otherwise no effect on the evolution of
  the performance function $h_k$ during the inter-reception
  times. Hence, from the definition~\eqref{eq:H-def} of $H $, it
  follows that
  \begin{equation*}
    \condexpect[\ETP]{ \pf_k }{ I_{S_j}^+ } = H(k - S_j, x_{S_j}^2),
    \quad \forall k \in \intrangecc{S_j}{S_{j+1}} .
  \end{equation*}
  Consequently, Proposition~\ref{prop:H} implies claim (a).

  Next, we prove claim (b). From Proposition~\ref{prop:GkD-prop}(a),
  we see that for all $k \in \intrangeoo{ S_j }{ S_{j+1} }$,
  \begin{align}
    \condexpect[\ETP]{ \G{k+1}{D} }{ I_{S_j}^+ } &= \condexpect[\ETP]{
      \condexpect[\ETP]{ \G{k+1}{D} }{ I_k, r_k = 0 } }{ I_{S_j}^+ }
    \notag
    \\
    &= \condexpect[\ETP]{ \G{k}{D+1} }{ I_{S_j}^+ } , \label{eq:step1}
  \end{align}
  and
  \begin{align}
    \label{eq:step2}
    \condexpect[\ETP]{ \G{S_j+1}{D} }{ I_{S_j}^+ } &=
    \condexpect[\ETP]{ \G{S_j+1}{D} }{ I_{S_j}, r_{S_j} = 1 } =
    \Gp{S_j}{D+1} .
  \end{align}
  Then, under the policy $\ETP$, and using~\eqref{eq:Tj},
  \begin{align*}
    \condexpect[\ETP]{ \pf_{ S_{j+1} } }{ I_{S_j}^+ } 
    &=
      \condexpect[\ETP]{ \condexpect[\ETP]{ \pf_{ S_{j+1} } }{
      I_{T_j} } }{ I_{S_j}^+ }
    \\
    & = \condexpect[\ETP]{ \condexpect[\Tc_{T_j}^0]{
      \pf_{ S_{j+1}
      } }{ I_{T_j} } }{ I_{S_j}^+ }
    \\
    &= \condexpect[\ETP]{ G_{T_j}^0 }{ I_{S_j}^+ } ,
  \end{align*}
  where we have first used the `Tower property' of conditional
  expectation, then the definition of the event-triggered
  policy~\eqref{eq:ETP} and finally the definition of
  $\G{T_j}{0}$. Using~\eqref{eq:step1} and~\eqref{eq:step2}
  recursively, this expression reduces to
  \begin{align*}
    &\condexpect[\ETP]{ \pf_{ S_{j+1} } }{ I_{S_j}^+ }
    \\
    &=
    \begin{cases}
      \Gp{S_j}{T_j - S_j}, \quad &\text{if } T_j \leq S_j + D
      \\
      \condexpect[\ETP]{ G_{T_j - D}^{D} }{ I_{S_j}^+ }, \quad
      &\text{if } T_j > S_j + D .
    \end{cases}
  \end{align*}
  In the case when $T_j \leq S_j + D$, claims (b) and (c) of
  Proposition~\ref{prop:GkD-prop} imply that
  $\Gp{S_j}{T_j - S_j} < 0$. Also note that, under the policy $\ETP$,
  $G_{k}^{D} < 0$ for all $k \in \intrangeoo{ S_j }{ T_j }$. Thus, in
  the case when $T_j > S_j + D$, we have $G_{T_j - D}^{D} < 0$. Thus,
  we have shown that claim (b) is true, which completes the proof.
\end{proof}

A consequence of Theorem~\ref{thm:main} along with
Lemma~\ref{lem:objectives} is that the event-triggered policy $\ETP$
guarantees
\begin{equation*}
  \condexpect[\ETP]{x_k^2}{I_0^+} \leq \max \{ c^{2k} x_0^2 , B \},
  \quad \forall k \in \integersnonneg ,
\end{equation*}
the original control objective. In other words, the proposed
event-triggered transmission policy guarantees that the expected value
of $x_k^2$ converges at an exponential rate to its ultimate bound of~$B$.

  \begin{remark}\longthmtitle{Sufficient
      conditions impose  lower bounds on the ultimate bound} \label{rem:suff-cond}
    {\rm 
      The two conditions identified in Theorem~\ref{thm:main} to
      ensure the satisfaction of the control objective may be
      interpreted as lower bounds on the choice of the ultimate
      bound~$B$. The first condition, $B > B^*$, comes from
      Proposition~\ref{prop:H} and ensures the monotonicity property
      of the open-loop performance function~$H$. Thus, as expected, we
      see from~\eqref{eq:Bc} that $B^*$ does not depend on the channel
      properties, namely the packet-drop probability $(1-p)$, or the
      look-ahead horizon~$D$.  On the other hand, the second
      condition, $\Gc(D) < 0$, imposes a lower bound on~$B$ which does
      have a dependence on both parameters $p$ and $D$. This condition
      can be rewritten as
      \begin{equation}\label{eq:rewrite-Gc}
        \left( g_D(c^2) - g_D( \abar^2 ) \right) \frac{ B }{ c^{2D} }
        > \Mbar \left( g_D(a^2) - g_D(1) \right).
      \end{equation}
      Comparing this inequality with~\eqref{eq:Bc}, we see that there
      is a strong resemblance between the two. In fact, if
      in~\eqref{eq:rewrite-Gc} the function $g_D$ were replaced with
      $\log$ and the factor $c^{2D}$ removed, we would
      get~\eqref{eq:Bc}. This is not unexpected because $\Gc(D)$ is a
      uniform (over the plant state space) upper bound on
      $\Gp{S_j}{D}$, which is nothing but the expectation of the
      open-loop performance function at the next (random) reception
      time.} \oprocend
  \end{remark}


  The next result shows that if the event-triggered transmission
  policy meets the control objective for a certain look-ahead horizon,
  then it also meets it for any other shorter look-ahead horizon.

\begin{corollary}\longthmtitle{If $\ETP$ meets the control objective
    with parameter $D$ then it also does with a smaller
    $D$}\label{cor:ETP-smaller-D}
  Let $B > B^*$ and $D \in \integerspos$ such that $\Gc(D) < 0$. Then,
  for any $D'<D$, the event-triggered transmission policy $\ETP$ with
  parameter $D'$ meets the control objective.
\end{corollary}
\begin{proof}
  For $D'<D$, Corollary~\ref{cor:Gc} guarantees that
  $\Gc(D')<\Gc(D)$. Therefore, $\Gc(D)<0$ implies that $\Gc(D')<0$.
  Thus, according to Theorem~\ref{thm:main}, the event-triggered
  transmission policy~\eqref{eq:ET-design} with parameter $D'$ ensures
  the control objective is met.
\end{proof}

Note that this result is aligned with Remark~\ref{rem:role-D} where we
made the observation that, intuitively, a larger~$D$ in the
event-triggered transmission policy~\eqref{eq:ET-design} is more
conservative.  It is also interesting to observe that, as a result of
Corollary~\ref{cor:ETP-smaller-D}, if $\Gc(D)<0$ is satisfied for $D >
1$, then the control objective is met with $D = 1$, which corresponds
to a time-triggered policy that transmits at every time step (i.e.,
has period $T = 1$).

\subsection{Performance guarantees: benefits over time-triggering}

Here we analyze the efficiency of the proposed event-triggered
transmission policy in terms of the fraction of the number of time
steps at which transmissions occur. For any stopping time $K$, we
introduce the \emph{expected transmission fraction}
\begin{equation}
  \label{eq:tx-frac}
  \Fc_0^K \triangleq \frac{ \condexpect[\ETP]{ \displaystyle
      \sum_{k = 1}^K \indfun{\{t_k = 1\}} }{ I_0^+ } }{ \condexpect[\ETP]{
      K }{ I_0^+ } } .
\end{equation}
This corresponds to the expected fraction of time steps from $1$ to
$K$ at which transmissions occur. Note that $K$ might be a random
variable itself, which justifies the expectation operation taken in
the denominator. The following result provides an upper bound on this
expected transmission fraction.

\begin{proposition}\longthmtitle{Upper bound on the expected
    transmission fraction}\label{prop:TF}
  Suppose that $\Gc(D+\Bc) < 0$ (see~\eqref{eq:param-constraint}),
  where $D$ is the parameter in the event-triggered transmission
  policy~\eqref{eq:ET-design} and $\Bc \in \integersnonneg$. Then
  \begin{equation*}
    \Fc_0^\infty \leq \frac{ 1 }{ 1 + \Bc p } .
  \end{equation*}
\end{proposition}
\begin{proof}
  Given Corollary~\ref{cor:ETP-smaller-D}, the remainder of the proof
  relies on finding an upper bound on the expected transmission
  fraction in a cycle from one reception time to the next, i.e.,
  $\Fc_{S_j}^{S_{j+1}}$ and then extending it to obtain the
  \emph{running transmission fraction} $\Fc_{0}^{S_N}$ for an
  arbitrary $N \in \integerspos$.  Note that, in any such cycle, the
  channel is idle, i.e., $t_k = 0$, for $k$ from $S_j + 1$ to
  $T_j - 1$ and transmissions occur from $T_j$ to $S_{j+1}$.

  Now, observe that the assumption that~\eqref{eq:param-constraint} is
  satisfied with $D + \Bc$ in place of $D$ implies, according to
  Proposition~\ref{prop:GkD-prop}(b), that $\Gp{S_j}{D+\Bc} < 0$ for
  all $j \in \integersnonneg$. We also know from~\eqref{eq:step1}
  and~\eqref{eq:step2} that
  \begin{align*}
    \condexpect[\ETP]{ \G{S_j+\Bc}{D} }{ I_{S_j}^+ } =
    \condexpect[\ETP]{ \Gp{S_j}{D+\Bc} }{ I_{S_j}^+ } < 0,
  \end{align*}
  which means that $T_j - 1 \geq S_j + \Bc$. Thus,
  \begin{align*}
    \condexpect[\ETP]{ \displaystyle \sum_{k=S_j + 1}^{S_{j+1}}
    \indfun{\{t_k = 0\}} }{ I_{S_j}^+ } &= \condexpect[\ETP]{
                                          T_j - 1 - S_j }{ I_{S_j}^+ }
                                          \geq \Bc .
  \end{align*}
  On the other hand, since the probability of $S_{j+1} - T_j + 1$
  being $1$ is $p$, being $2$ is $(1-p) p$, and so on, we note that
  \begin{align*}
    \condexpect[\ETP]{ \displaystyle \sum_{k=S_j + 1}^{S_{j+1}}
    \indfun{\{t_k = 1\}} }{ I_{S_j}^+ } &= \condexpect[\ETP]{S_{j+1} -
                                          T_j + 1 }{ I_{S_j}^+ }
    \\
                                        &= p \sum_{s=0}^\infty (s + 1)
                                          (1-p)^s = \frac{ 1 }{ p } .
  \end{align*}
  We can extend this reasoning further to $N$ cycles, from $S_0 = 0$
  to $S_N$, to obtain
  \begin{align*}
    &\condexpect[\ETP]{ \displaystyle \sum_{k=1}^{S_N} \indfun{\{t_k =
        0\}} }{ I_{S_j}^+ } \geq N \Bc, 
    \\
    &\condexpect[\ETP]{ \displaystyle \sum_{k=1}^{S_N} \indfun{\{t_k =
        1\}} }{ I_{S_j}^+ } = \frac{ N }{ p } .
  \end{align*}
  Finally, note that
  \begin{align*}
    &\condexpect[\ETP]{ S_N }{ I_0^+ }
    \\
    &= \condexpect[\ETP]{ \displaystyle \sum_{k=1}^{S_N} \indfun{\{t_k =
        0\}}}{ I_0^+ } + \condexpect[\ETP]{ \displaystyle
      \sum_{k=1}^{S_N} \indfun{\{t_k = 1\}} }{ I_0^+ } .
  \end{align*}
  Then using~\eqref{eq:tx-frac}, this yields an upper bound on the
  expected transmission fraction during $\intrangecc{0}{S_N}$
  \begin{equation*}
    \Fc_0^{S_N} \leq \frac{ 1 }{ 1 + \Bc p } ,
  \end{equation*}
  which we see is independent of $N$. The result then follows by
  taking the limit as $N \rightarrow \infty$.
\end{proof}

An expected transmission fraction of $1$ corresponds to a transmission
occurring at every time step almost surely, i.e., essentially a
time-triggered policy. Therefore, Proposition~\ref{prop:TF} states
that the number of transmissions under the event-triggered policy
$\ETP$ is guaranteed to be less than that of a time-triggered policy.

  \begin{remark}\longthmtitle{Interpretation of the parameter
      $D$-cont'd} \label{rem:role-D-continued}
    {\rm Proposition~\ref{prop:TF} is consistent with our intuition,
      cf. Remark~\ref{rem:role-D}, that a larger~$D$ in the
      event-triggered transmission policy~\eqref{eq:ET-design} is more
      conservative. In fact, if $D_1 < D_2$ and $D_1 + \Bc_1 = D_2 +
      \Bc_2$, then $\Bc_1 > \Bc_2$ and thus the upper bound on the
      expected transmission fraction is larger for larger~$D$. Note
      that since Proposition~\ref{prop:TF} is only a statement about
      the upper bound on the expected transmission fraction, we do not
      formally claim that larger $D$ is more conservative. In fact,
      different control parameters lead to different state
      trajectories and thus, formally, we can only say that larger $D$
      is more conservative at each point in state space (corresponding
      to the initial condition for each trajectory). However,
      Remark~\ref{rem:role-D}, Corollary~\ref{cor:ETP-smaller-D},
      Proposition~\ref{prop:TF} and the simulation results in the
      sequel together suggest that, starting from the same initial
      conditions, a larger $D$ has a larger expected transmission
      fraction.  \oprocend }
  \end{remark}

\begin{remark}\longthmtitle{Optimal sufficient periodic transmission
    policy}\label{rem:optim-periodic-policy} {\rm Under the
    assumptions of Proposition~\ref{prop:TF}, we know that the
    time-triggered policy with period $T=1$ satisfies the control
    objective.  It is conceivable that a time-triggered transmission
    policy with period $T > 1$ (i.e., with transmission fraction
    $1/T<1$) also achieves it. To see this, consider the open-loop
    performance evolution function~\eqref{eq:H-form} at integer
    multiples of $T$, i.e., $H(sT, y)$
    and~\eqref{eq:param-constraint}. Then, a time-triggered
    transmission policy with period $T$ achieves the control objective
    if \begin{equation}\label{eq:periodic-policy-suff-cond}
      \mkern-10mu \left( g_1( \abar^{2T} ) - g_1(c^{2T}) \right)
      \frac{ B }{ c^{2T} } + \Mbar \left( g_1(a^{2T}) - \frac{1}{p}
      \right) < 0 .
    \end{equation}
    The periodic transmission policy with the least transmission
    fraction can be found by maximizing $T$ that
    satisfies~\eqref{eq:periodic-policy-suff-cond}. In any case, a
    time-triggered implementation determines the transmission times a
    priori, while the event-triggered implementation determines them
    online, in a feedback fashion. The latter therefore renders the
    system more robust to uncertainties in the knowledge of the system
    parameters, noise and packet drop distributions.  }  \oprocend
\end{remark}

\section{Extension to the vector case}\label{sec:vector-case}

In this section, we outline how to extend the design and analysis of
the event-triggered transmission policy to the vector case, and
discuss the associated challenges. Consider a multi-dimensional system
evolving as

\begin{subequations}\label{eq:vec-sys}
  \begin{equation}\label{eq:vec-x}
    x_{k+1} = A x_k + Q u_k + v_k ,
  \end{equation}
  where $x \in \real^n$, $u \in \real^m$, $v \in \real^n$, $A \in
  \real^{n \times n}$, and $Q \in \real^{n \times m}$. The process
  noise $v$ is zero-mean independent and identically distributed with
  positive semi-definite covariance matrix $\Sigma$. Let the control
  be given by $u_k = L \xhat_k^+$, where $\xhat_k^+$ is given
  by~\eqref{eq:xhat_p} and
  \begin{equation}\label{eq:vec-xhat}
    \xhat_{k+1} = \Abar \xhat_k^+ \triangleq (A + QL) \xhat_k^+ .
  \end{equation}
\end{subequations}
We can define the performance function as
\begin{equation*}
  \pf_k = x_k^T x_k - \max \{ c^{2(k-\LRT{k})} x_{\LRT{k}}^T x_{\LRT{k}}, B \} .
\end{equation*}

The key to our developments of Sections~\ref{sec:design}
and~\ref{sec:analysis} is the explicit closed-form expressions of the
look-ahead criterion and the performance evaluation functions,
$\G{k}{D}$ and $\Gp{k}{D}$ respectively, which have allowed us to
evaluate the trigger~\eqref{eq:ET-design} and unveil the necessary
properties to ensure the satisfaction of the original control
objective. However, in the vector case, it is challenging to obtain
closed-form expressions for these functions because this involves
obtaining closed-form expressions for the series
\begin{equation*}
  \sum_{s=D}^\infty \nu_1^T (M_1^s)^T M_2^s \nu_2 ,
\end{equation*}
with $\nu_1$ and $\nu_2$ being vectors such as $x_k$, $e_k$ or
$v_{k+s}$, and $M_1$ and $M_2$ being either of the matrices $A$ or
$\Abar$.  It is however possible to obtain closed-form upper bounds
$\bar{G}_k^D$ and $\bar{J}_k^D$ for the functions $\G{k}{D}$ and
$\Gp{k}{D}$ respectively, essentially by upper bounding $\expect{
  \pf_{k+s} }$. The following result makes this explicit.

\begin{proposition}\longthmtitle{Upper bound on the expected value of the
    performance function}\label{prop:vec-hbound}
  For $k \in \integersnonneg$, let $s \in \integersnonneg$ be such
  that $k + s \in \intrangeoc{\LRT{k}}{\LRT{k + 1}}$ and define
  \begin{align*}
    &\pfb_{k+s} \triangleq \Enorm{ \Abar }^{2s} \Enorm{x_k}^2 + 2
    \Enorm{ \Abar }^s (\Enorm{ A }^s + \Enorm{ \Abar }^s) \Enorm{x_k}
    \Enorm{e_k}
    \\
    &+ ( \Enorm{ A }^{2s} + 2 \Enorm{ A }^s \Enorm{ \Abar }^s +
    \Enorm{ \Abar }^{2s} ) \Enorm{ e_k }^2 + \Mbar ( \Enorm{ A
    }^{2s} - 1 ) \\
    &- \max \{ c^{2(k+s-\LRT{k})} x_{\LRT{k}}^T x_{\LRT{k}}, B \} ,
  \end{align*}
  where $\Mbar \triangleq \frac{ \tr{\Sigma} }{ \Enorm{A}^2 - 1
  }$. Further, for $k = S_j$ with $j \in \integersnonneg$, define
  \begin{align*}
    &\pfb_{k+s}^+ \triangleq \Enorm{ \Abar }^{2s} \Enorm{x_{k}}^2 +
    \Mbar ( \Enorm{ A }^{2s} - 1 ) - \max \{ c^{2s} x_{k}^T x_{k}, B
    \} .
  \end{align*}
  Then,
  \begin{align*}
    \condexpect{ \pf_{k+s} }{ I_k, k + s = \LRT{k+1} } & \leq \expect{
      \pfb_{k+s} } = \pfb_{k+s} ,
    \\
    \condexpect{ \pf_{S_j+s} }{ I_{S_j}^+, S_j + s = S_{j+1} } & \leq
    \expect{ \pfb_{S_j+s}^+ } = \pfb_{S_j+s}^+ .
  \end{align*}
\end{proposition}
\begin{proof}
  Notice from~\eqref{eq:vec-sys} that
  \begin{align*}
    x_{k+s} = A^s x_k + \sum_{i=0}^{s-1} A^{s-1-i}(\Abar - A)\Abar^i
    \xhat_k^+ + \sum_{i=0}^{s-1} A^{s-1-i} v_{k+i} .
  \end{align*}
  Further, observe that
  \begin{align*}
    \sum_{i=0}^{s-1} A^{s-1-i}(\Abar - A)\Abar^i &= \sum_{i=1}^{s}
    A^{s-i}\Abar^i - \sum_{i=0}^{s-1} A^{s-i}\Abar^i \\
    &= \Abar^s - A^s ,
  \end{align*}
  which yields
  \begin{align*}
    x_{k+s} = A^s x_k + (A^s - \Abar^s) e_k^+ + \sum_{i=0}^{s-1}
    A^{s-1-i} v_{k+i} ,
  \end{align*}
  where we have also used the fact that $e_k^+ = x_k -
  \xhat_k^+$. Then,
  \begin{align*}
    &\condexpect{ x_{k+s}^T x_{k+s} }{ I_k^+, \LRT{k+1} = k + s } \\
    &= x_k^T (\Abar^s)^T \Abar^s x_k + 2 x_k^T (\Abar^s)^T (A^s -
    \Abar^s) e_k^+ \\
    &+e_k^+ (A^s - \Abar^s)^T (A^s - \Abar^s) e_k^+\\
    &+ \expect{ \sum_{i=0}^{s-1} v_{k+i}^T ( A^{s-1-i} )^T A^{s-1-i}
      v_{k+i} } \\
    &\leq \Enorm{ \Abar }^{2s} \Enorm{x_k}^2 + 2 \Enorm{ \Abar }^s
    (\Enorm{ A }^s + \Enorm{ \Abar }^s) \Enorm{x_k} \Enorm{e_k^+} \\
    &+ ( \Enorm{ A }^{2s} + 2 \Enorm{ A }^s \Enorm{ \Abar }^s +
    \Enorm{ \Abar }^{2s} ) \Enorm{ e_k^+ }^2 + \Mbar ( \Enorm{ A
    }^{2s} - 1 ) .
  \end{align*}
  The result now follows from the fact that $e_k^+ = e_k$ if $k \neq
  S_j$ for any $j \in \integersnonneg$ and the definitions of $\pf_k$,
  $\pfb_{k+s}$ and $\pfb^+_{k+s}$.
\end{proof}

Based on Proposition~\ref{prop:vec-hbound}, we define $\bar{G}_k^D$
and $\bar{J}_k^D$ analogously to~\eqref{eq:GkD} and~\eqref{eq:Gp-kD},
respectively, except with $\pfb$ and $\pfb^+$ instead of~$\pf$. We do
not include the resulting closed-form expressions of $\bar{G}_k^D$ and
$\bar{J}_k^D$ for the sake of brevity. 

With these elements in place, we define the event-triggered policy
$\ETPvec$ given the last successful reception time $R_k = S_j$
as
\begin{subequations}\label{eq:ET-vec}
  \begin{equation}
    \label{eq:ETP-vec}
    \ETPvec : t_k =
    \begin{cases}
      0, \quad \text{if } k \in \{ S_j + 1, \ldots, F_k - 1 \}
      \\
      1, \quad \text{if } k \in \{ F_k, \ldots, S_{j+1} \} ,
    \end{cases}
  \end{equation}
  where
  \begin{equation}\label{eq:Fk-vec}
    F_k \triangleq \min \{ \ell > R_k : \bar{G}_\ell^D \geq 0 \} .
  \end{equation}
\end{subequations}
The next result establishes that this new event-triggered policy
guarantees the desired stability result in the vector case.

\begin{theorem}\longthmtitle{The event-triggered policy meets the
    control objective}\label{thm:main-vc}
  Suppose $D \in \integerspos$ and that~\eqref{eq:param-constraint} is
  satisfied with $a = \Enorm{A}$ and $\abar = \Enorm{\Abar}$. Then,
  under the same hypotheses as in Proposition~\ref{prop:H}, the
  event-triggered policy $\ETPvec$ guarantees that
  $\condexpect[\ETPvec]{ \pf_{k} }{ I_{\LRT{k}}^+ } \leq 0$ for all $k
  \in \integerspos$.
\end{theorem}
\begin{proof}
  The main step is in proving an upper bound analogue of
  Proposition~\ref{prop:GkD-prop}(a). Note that
  \begin{subequations}\label{eq:weaker-properties}
  \begin{align}
    &\condexpect[\Tc]{ \G{k+1}{D} }{ I_k, r_k = 0 } \notag
    \\
    &= \condexpect[\Tc]{ \condexpect[\Tc_k^D]{ h_{ \LRT{k+1} } }{ I_k
      } }{ I_k, r_k = 0 } \notag
    \\
    &\leq \condexpect[\Tc]{ \pfb_{ \LRT{k+1} } }{ I_k, r_k = 0 }
    =\bar{G}_{k}^{D+1} ,
  \end{align}
  where we have used Proposition~\ref{prop:vec-hbound} in the
  inequality.  A similar reasoning yields
  \begin{align}
    &\condexpect[\Tc]{ \G{k+1}{D} }{ I_k, r_k = 1 } \leq
    \bar{J}_{k}^{D+1} .
  \end{align}    
  \end{subequations}
  Notice from the definition of $\pfb_{S_j+s}^+$ in
  Proposition~\ref{prop:vec-hbound} that the expression for
  $\bar{J}_{S_j}^{D}$ is the same as that of $\Gp{S_j}{D}$ in
  Lemma~\ref{lem:closed-form-Gp} with $a = \Enorm{A}$ and $\abar =
  \Enorm{\Abar}$. As a result, in the vector case, claims (b) and (c)
  of Proposition~\ref{prop:GkD-prop} hold for $\bar{J}_{S_j}^{D}$. The
  rest of the proof follows along the lines of the proof of
  Theorem~\ref{thm:main}.
\end{proof}

Thus, the upper bounds~\eqref{eq:weaker-properties} relating
$\G{k+1}{D}$ to $\bar{G}_k^{D+1}$ and $\bar{J}_k^{D+1}$ are sufficient
to guarantee that the event-triggered policy~\eqref{eq:ET-vec} meets
the control objective. However, the lack of a relationship between
$\bar{G}_{k}^{D+1}$ and $\bar{G}_{k}^{D}$ or $\bar{J}_{k}^{D}$
prevents us from obtaining an upper bound on the expected transmission
fraction. Nonetheless, as we described in
Remark~\ref{rem:optim-periodic-policy}, given the fact that a
time-triggering sampling period can only be designed keeping the worst
case in mind, it is reasonable to expect that the event-triggered
transmission policy would be more efficient in the usage of the
communication channel (this is shown in the simulations of the next
section). Finally, we believe that, in order to analytically quantify
transmission fraction and assess the efficiency of the event-triggered
design, one needs to make more substantial modifications to the
definitions of the functions $\bar{G}_k^D$ and~$\bar{J}_k^D$.

\myclearpage
\section{Simulations}\label{sec:sims}

Here we present simulation results for the system evolution under the
event-triggered transmission policy $\ETP$, first for a scalar system,
and then a vector system.

\subsubsection*{Scalar system}
We consider the dynamics~\eqref{eq:sys-evolve} with the following
parameters,
\begin{align*}
  &a = 1.05, \ p = 0.6, \ M = 1, \ c = 0.98, \ \abar = 0.95c,
  \\
  &B = 15.5, \ x(0) = 10 B .
\end{align*}
The process noise is drawn from a Gaussian distribution, with
covariance $M$.
To find the critical value $B^* = 12.92$ in Proposition~\ref{prop:H},
we use the method described in the proof of Lemma~\ref{lem:Bstar} and
the discussion subsequent to it. We performed simulations for $1000$
realizations of process noise and packet drops, all starting from the
same initial condition. Then, for each time step $k$, we computed the
empirical mean of the various quantities. This is illustrated in
Figures~\ref{fig:ETPD-sim} and~\ref{fig:ETPD-NT}. We performed
simulations with $D = 1$ and $D = 3$, and in each case $D + \Bc =
3$. Figure~\ref{fig:ETPD-sim} shows that the control
objective~\eqref{eq:overall-objective} is satisfied, as guaranteed by
Theorem~\ref{thm:main}. For $D = 3$, one can see that the control
objective is met more conservatively, which is consistent with the
intuitive interpretation of the transmission policies given in
Section~\ref{sec:two-step-design}.
\begin{figure}[htb!]
  \centering
  \includegraphics[width=.875\linewidth]{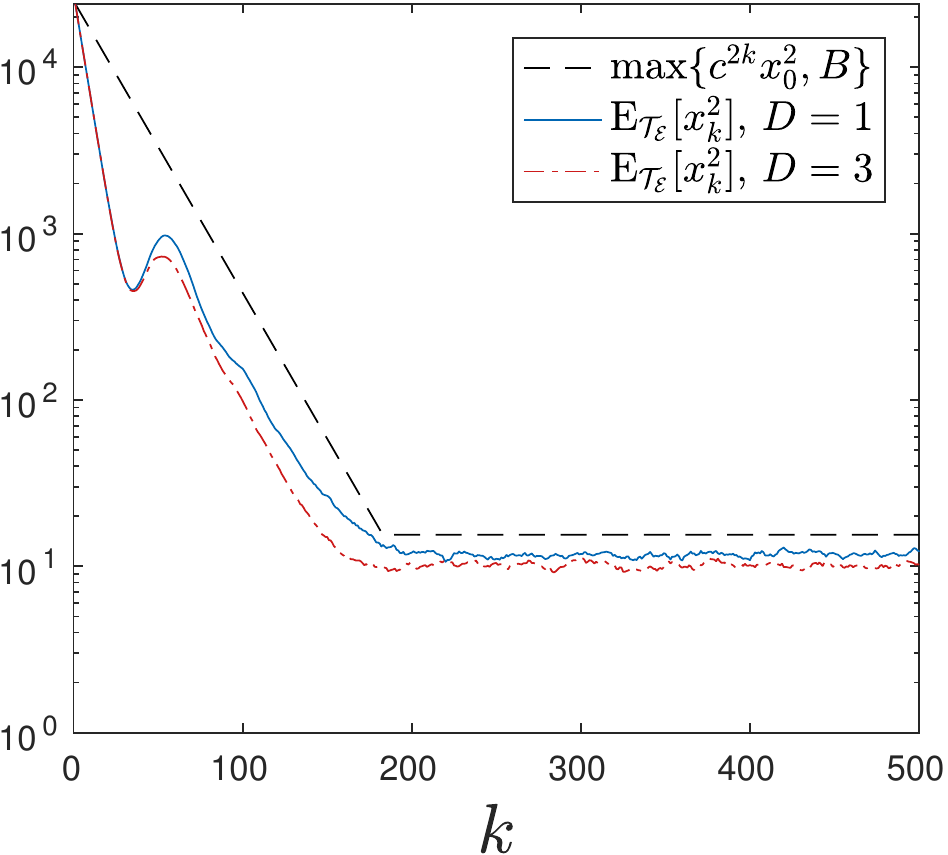}
  \caption{Plot of the evolution of the empirical mean
    $\expect[\ETP]{x_k^2}$ for the scalar example under the
    event-triggered transmission policy~\eqref{eq:ET-design} with $D =
    1$ and $D = 3$ and the performance bound, $\max \{ c^{2k} x_0^2 ,
    B \}$.}\label{fig:ETPD-sim}
\end{figure}
Figure~\ref{fig:ETPD-NT} shows the empirical running transmission
fractions for $D = 1$ and $D =3$, as well as the upper bound on the
transmission fraction $\Fc_0^\infty$ in the case of $D = 1$ obtained
in Proposition~\ref{prop:TF}. In the case of $D = 3$, this quantity
is~$1$.
\begin{figure}[htb!]
  \centering
  \includegraphics[width=.875\linewidth]{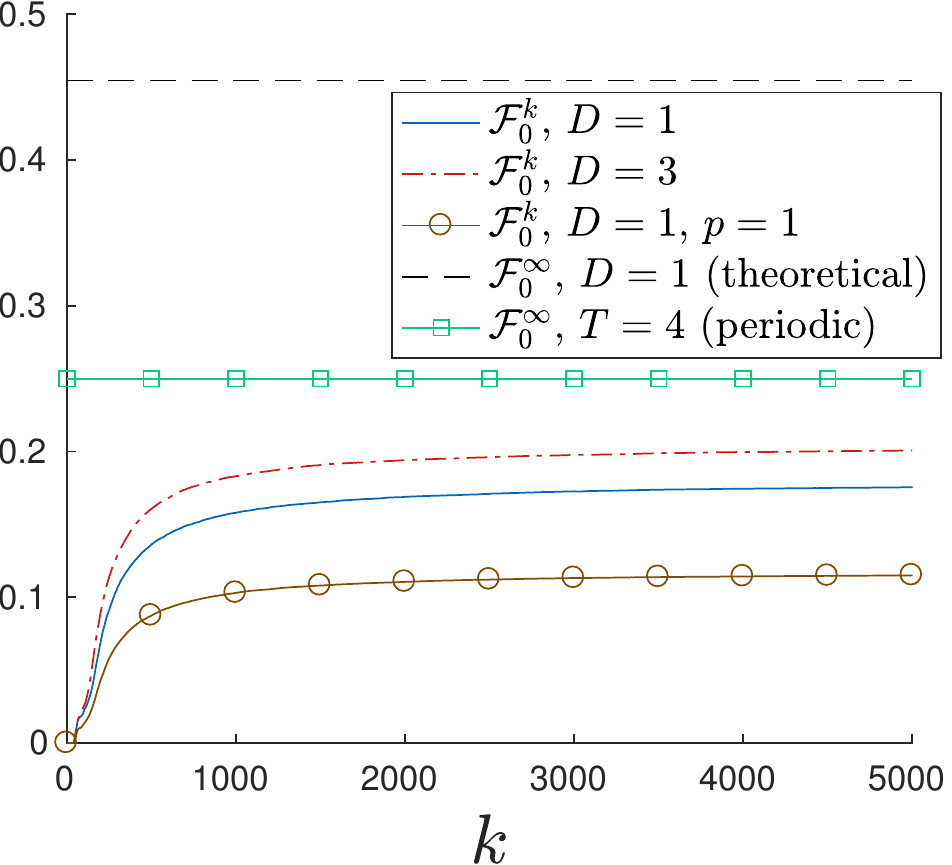}
  \caption{Plot of the evolution of the empirical running transmission
    fraction~$\Fc_0^k$ for the scalar example under the
    event-triggered transmission policy~\eqref{eq:ET-design} for
    $D = 1$ and $D = 3$, and the theoretical bound on the asymptotic
    transmission fraction $\Fc_0^\infty$ in the case of $D = 1$
    obtained in Proposition~\ref{prop:TF}. For $D = 3$, the latter
    is~$1$. For comparison, the plot also shows the empirical running
    transmission fractions in the case of $D = 1$ and $p =1$ (perfect
    channel) and for a periodic policy with period $T = 4$.
  }\label{fig:ETPD-NT}
\end{figure}
As expected, the conservativeness of the implementation with $D=3$ is
reflected in a higher transmission fraction and in the
conservativeness with which the control objective is satisfied. We
found the optimal sufficient period for a periodic transmission
policy, cf.  Remark~\ref{rem:optim-periodic-policy}, to be $1$. Thus,
the transmission fraction for the optimal sufficient periodic
transmission policy is $1$, which is higher than both the theoretical
and the actual transmission fractions for our implementation with
$D = 1$. Figure~\ref{fig:ETPD-NT} also shows the running transmission
fraction in the case of $D =1$ and $p = 1$ (perfect channel) and for a
periodic policy with period $T = 4$. We see that the proposed policy
automatically adjusts its transmission fraction with changes in the
dropout probability. Figure~\ref{fig:TTPD-sim} shows the evolution of
the performance function under the periodic policy with period $T = 4$
in the case of a deterministic channel ($p = 1$) and with a dropout
probability of $(1-p)=0.4$.  This is an example of a policy that works
for a perfect channel ($p=1$) but does not work for an imperfect one
($p = 0.6$).  Although the transmission fraction for this policy is
higher than that of our policy (cf. Figure~\ref{fig:ETPD-NT}), it
still fails to meet the control objective in the case of $p = 0.6$,
demonstrating the usefulness of the proposed event-triggered policy
over a periodic policy.
\begin{figure}[htb!]
  \centering
  \includegraphics[width=.875\linewidth]{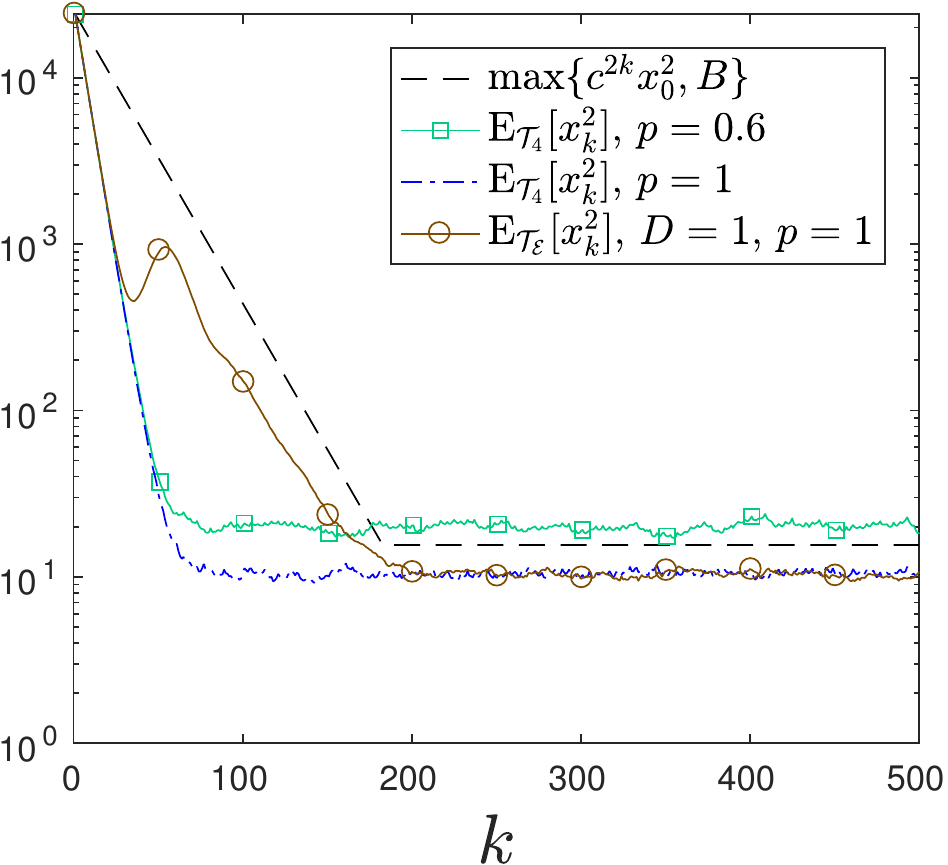}
  \caption{Plot of the evolution of the empirical mean
    $\expect[\mathcal{T}_4]{x_k^2}$ for the scalar example under the
    periodic transmission policy with period $T = 4$ with $p = 1$
    (deterministic channel) and $p = 0.6$ and the performance bound,
    $\max \{ c^{2k} x_0^2 , B \}$. Also shown is the empirical mean
    $\expect[\ETP]{x_k^2}$ under the event-triggered transmission
    policy~\eqref{eq:ET-design} with $D = 1$ and
    $p = 1$.}\label{fig:TTPD-sim}
\end{figure}
Finally, Figure~\ref{fig:tr-proc} shows sample transmission and
reception sequences for the event-triggered policy with $D = 1$ and
$p = 0.6$ and $p = 1$. We also show corresponding sequences for a
periodic policy with period $T = 4$ and $p = 0.6$. The plots show an
arbitrarily chosen interval of $50$ time steps for each transmission
and reception to be clearly distinguished.
\begin{figure}[htb!]
  \centering
  \subfigure[Event-triggered, $D = 1$, $p = 0.6$\label{fig:tr-ET06}]{\includegraphics[width=0.4\textwidth]{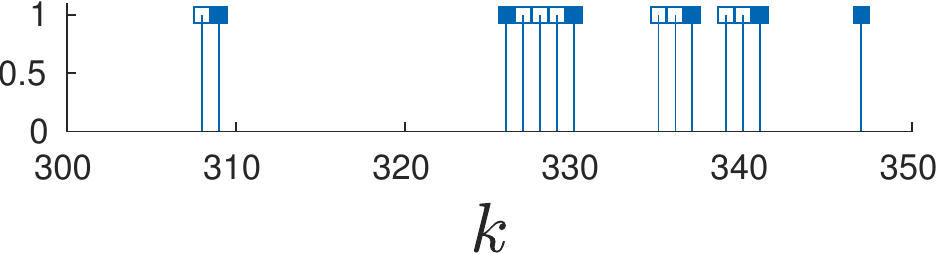}}\\
  \subfigure[Event-triggered, $D = 1$, $p = 1$\label{fig:tr-ET1}]{\includegraphics[width=0.4\textwidth]{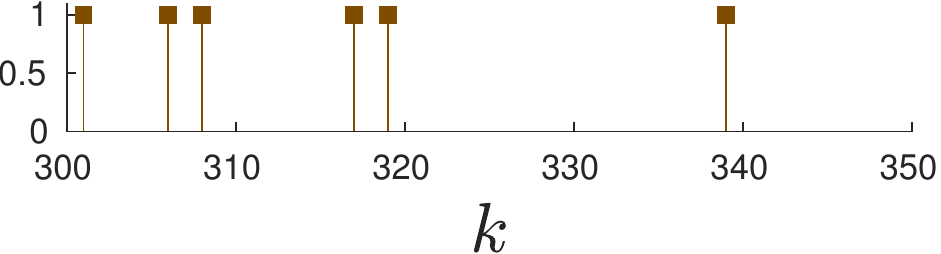}}\\
  \subfigure[Periodic, $T = 1$,
  $p =
  0.6$\label{fig:tr-TT06}]{\includegraphics[width=0.4\textwidth]{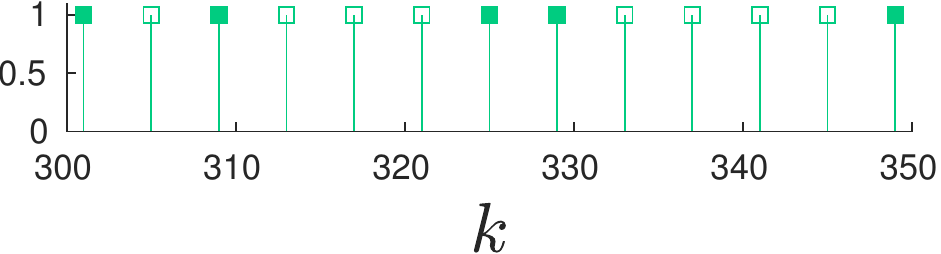}}
  \caption{Sample transmission and reception sequences for the
    event-triggered transmission policy~\eqref{eq:ET-design} with (a)
    $D = 1$ under $p = 0.6$ and (b) $D = 1$ under $p = 1$ and for (c)
    a periodic policy with period $T = 4$ and under $p = 0.6$. In each
    plot, an empty square corresponds to a ``transmission, but no
    reception'' and a filled square corresponds to a ``transmission
    and successful reception''.}\label{fig:tr-proc}
\end{figure}

\subsubsection*{Vector system}
We consider the dynamics~\eqref{eq:vec-sys} with the following system
matrices
\begin{align*}
  &A =
  \begin{bmatrix}
    0.8 & 0.5
    \\
    -0.5 & 1
  \end{bmatrix},
  \ Q =
  \begin{bmatrix}
    1 & 0
    \\
    0 & 1
  \end{bmatrix},
  \\
  &L =
  \begin{bmatrix}
    0.1310 & -0.5000
    \\
    0.5000 & -1.8820
  \end{bmatrix},
  \ \Sigma =
  \begin{bmatrix}
    0.1000 & 0.0500
    \\
    0.0500 & 0.1000
  \end{bmatrix} ,
\end{align*}
and the parameters $p = 0.8$ and $c = 0.98$. For this system, we get
$B^* = 2.44$ and chose $B = 2.93$. The initial condition is $x(0) =
B.[10 -5]^T$. We performed the same number of simulations as in the
scalar example to compute the empirical mean of the various relevant
quantities.  The results of simulations under the event-triggered
transmission policy~\eqref{eq:ET-vec} are illustrated in
Figures~\ref{fig:ETPD-sim-vec}
and~\ref{fig:ETPD-NT-vec}. Figure~\ref{fig:ETPD-sim-vec} shows that
the control objective is met, as stated in
Theorem~\ref{thm:main-vc}. The conservativeness that results from the
use of the upper bounds from Proposition~\ref{prop:vec-hbound} in the
definition of the event-trigger criterium is quite apparent from the
gap between the control objective and the actual trajectories of
$\expect[\ETPvec]{ \pf_{k} }$ compared to
Figure~\ref{fig:ETPD-sim}. Figure~\ref{fig:ETPD-NT-vec} also shows
that, as in the scalar case, smaller $D$ results in a less
conservative and more efficient design.
\begin{figure}[htb!]
  \centering
  \includegraphics[width=.875\linewidth]{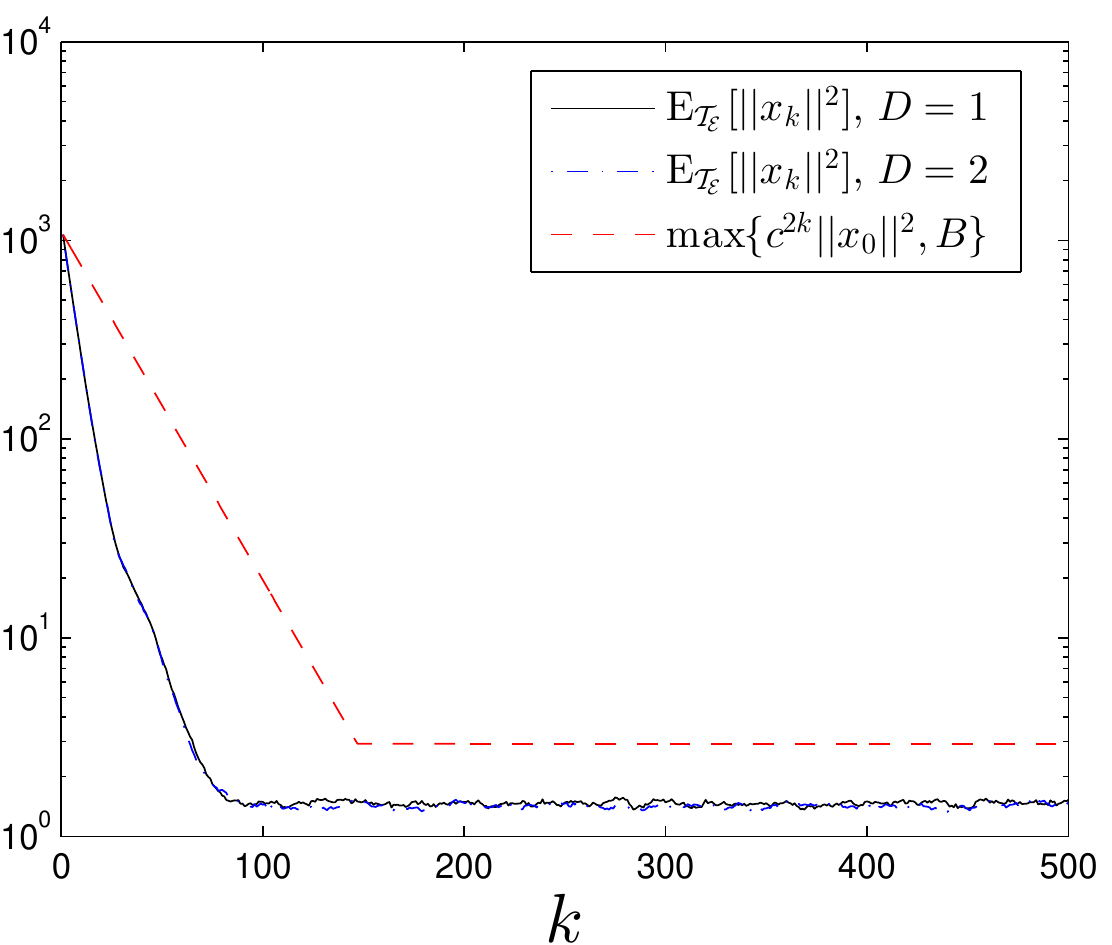}
  \caption{Plot of the evolution of the empirical mean
    $\expect[\ETP]{\norm{x_k}^2}$ for the vector example under the
    event-triggered transmission policy~\eqref{eq:ET-vec} with $D = 1$
    and $D = 2$ and the performance bound, $\max \{ c^{2k}
    \norm{x_0}^2 , B \}$.}\label{fig:ETPD-sim-vec}
\end{figure}
\begin{figure}[htb!]
  \centering
  \includegraphics[width=.875\linewidth]{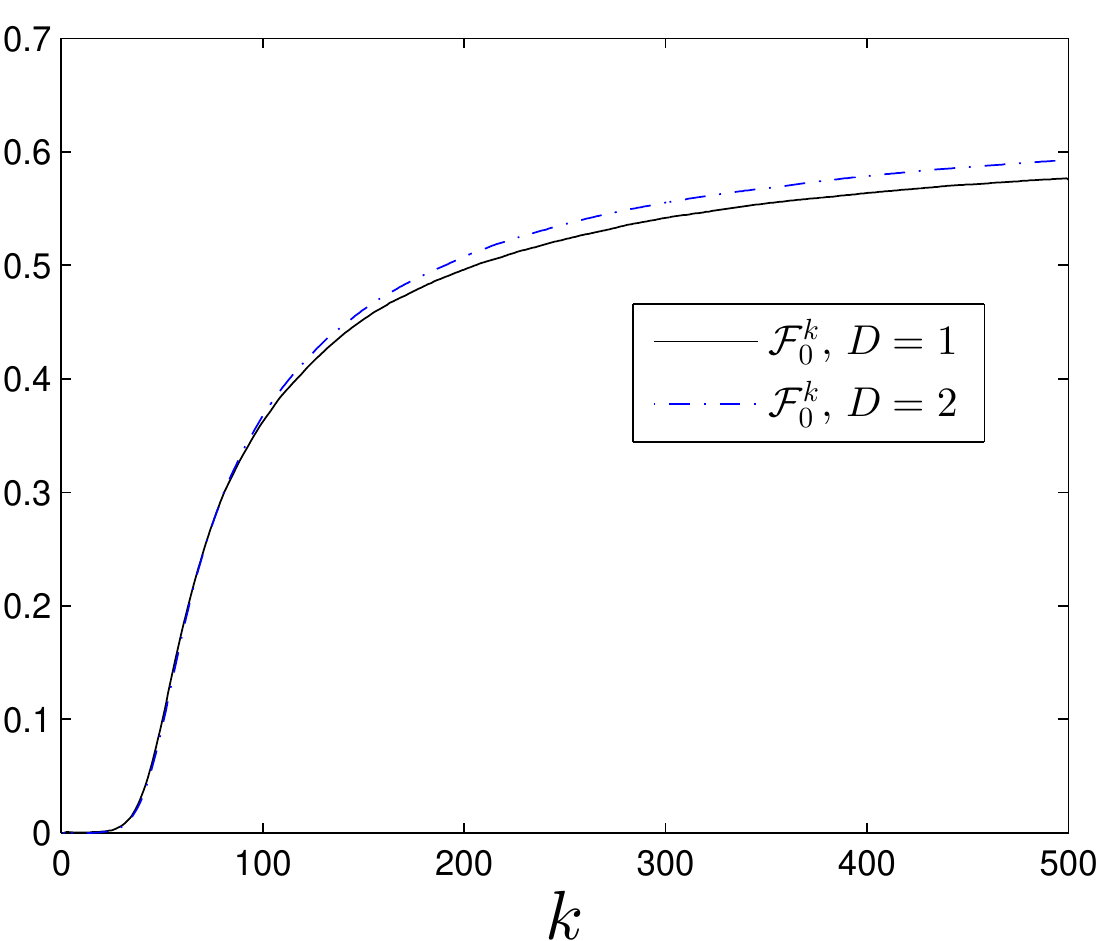}
  \caption{Plot of the evolution of the empirical running transmission
    fraction $\Fc_0^k$ for the vector example under the
    event-triggered transmission policy~\eqref{eq:ET-vec} for $D = 1$
    and $D = 2$.}\label{fig:ETPD-NT-vec}
\end{figure}

\section{Conclusions}\label{sec:conclusions}

We have designed an event-triggered transmission policy for scalar
linear systems under packet drops. The control objective consists of
achieving second-moment stability of the plant state with a given
exponential rate of convergence to an ultimate bound in finite
time. The synthesis of our policy is based on a two-step design
procedure. First, we consider a nominal quasi-time-triggered policy
where no transmission occurs for a given number of timesteps, and then
transmissions occur on every time step thereafter.  Second, we define
the event-trigger policy by evaluating the expectation of the system
performance at the next reception time given the current information
under the nominal policy, and prescribe a transmission if this
expectation does not meet the objective.  We have also characterized
the efficiency of our design by providing an upper bound on the
fraction of the expected number of transmissions over the infinite
time horizon. Finally, we have discussed the extension to the vector
case, and highlighted the challenges in characterizing the efficiency
of the event-triggered design. Future work will seek to address these
challenges in the vector case, incorporate measurement noise, output
measurements, lossy acknowledgments and will investigate the
possibilities for optimizing the two-step design of event-triggered
transmission policies, formally characterize the robustness advantages
of event-triggered versus time-triggered control, and investigate the
role of quantization and information-theoretic tools to address
questions about necessary and sufficient data rates.

\section*{Acknowledgments}
This work was supported in part by NSF Award CNS-1446891.

\bibliographystyle{IEEEtran}
\bibliography{alias,FB,JC,Main,Main-add}

\section*{Appendix: Glossary of symbols}

For the reader's reference, we present here a list of the symbols most
frequently used along the paper.

\begin{itemize}
\item State variables and functions
  \begin{itemize}
  \item $x_k$: plant state at time $k$
  \item $v_k$: process noise at time $k$
  \item $\xhat_k$: sensor's estimate of plant state at time $k$ given
    the `history' up to time $k-1$
  \item $\xhat_k^+$: controller's estimate of plant state at time $k$
    given `history' up to time $k$, including any reception at
    time~$k$
  \item $e_k \triangleq x_k - \xhat_k$: sensor estimation error
  \item $e_k^+ \triangleq x_k - \xhat_k^+$: controller estimation
    error
  \item $u_k \triangleq L\xhat_k^+$: control action at time $k$
  \item $I_k$: information available to the sensor at time $k$ before
    the decision to transmit or not
  \item $I_k^+$: information available to the controller at time $k$,
    which can also be computed by the sensor
  \item $h_k$: value of performance function at time $k$
  \end{itemize}

\medskip
\item System and performance parameters 
  \begin{itemize}
  \item $a$: open-loop `gain'
  \item $\Mbar \triangleq \frac{M}{a^2 -1}$: here $M$ is the
    covariance of $v_k$
  \item $\abar \triangleq a + L$: closed-loop `gain' in the case of
    perfect transmissions on all time steps
  \item $(1-p)$: probability of dropping a packet
  \item $B$: ultimate bound for second moment of plant state
  \item $c^2 \in (0,1)$: prescribed convergence rate for second moment
    of plant state
  \end{itemize}

\medskip
\item Transmission and reception process variables
  \begin{itemize}
  \item $t_k \in \{0,1\}$: no transmission/transmission at time $k$
  \item $r_k \in \{0,1\}$: no reception/reception at time $k$
  \item $R_k$: latest reception time before $k$
  \item $R_k^+$: latest reception time up to (including) $k$
  \item $S_j$: $j^{\text{th}}$ reception time
  \end{itemize}

\medskip
\item Symbols related to transmission policy
  \begin{itemize}
  \item $\Tc_k^D$: nominal transmission policy at time $k$ with
    parameter $D$
  \item $\G{k}{D}$: look-ahead criterion at time $k$ with parameter
    $D$
  \item $\ETP$: proposed event-triggered transmission policy
  \item $D$: `idle duration' in the nominal policy and `look-ahead
    horizon' in event-triggered policy
  \item $T_j$: first time a transmission occurs after $S_j$ under
    $\ETP$
  \item $\Gp{k}{D}$: performance-evaluation function at time $k$ with
    parameter $D$
  \item $H$: open-loop performance evolution function
  \item $\gD{b} \triangleq \frac{ b^D }{ 1 - b(1-p) }$
  \end{itemize}
\end{itemize}

\begin{IEEEbiography}[{\includegraphics[width=1in,
    height=1.25in,clip,keepaspectratio]
    {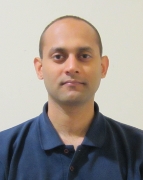}}]{Pavankumar Tallapragada}
  received the B.E. degree in Instrumentation Engineering from SGGS
  Institute of Engineering $\&$ Technology, Nanded, India in 2005,
  M.Sc. (Engg.) degree in Instrumentation from the Indian Institute of
  Science, Bangalore, India in 2007 and the Ph.D. degree in Mechanical
  Engineering from the University of Maryland, College Park in
  2013. He held a postdoctoral position at the University of
  California, San Diego during 2014 to 2017. He is currently an
  Assistant Professor in the Department of Electrical Engineering at
  the Indian Institute of Science, Bengaluru, India. His research
  interests include event-triggered control, networked control
  systems, distributed control and networked transportation systems.
\end{IEEEbiography}

\begin{IEEEbiography}
  [{\includegraphics[width=1in,height=1.25in,clip,keepaspectratio]{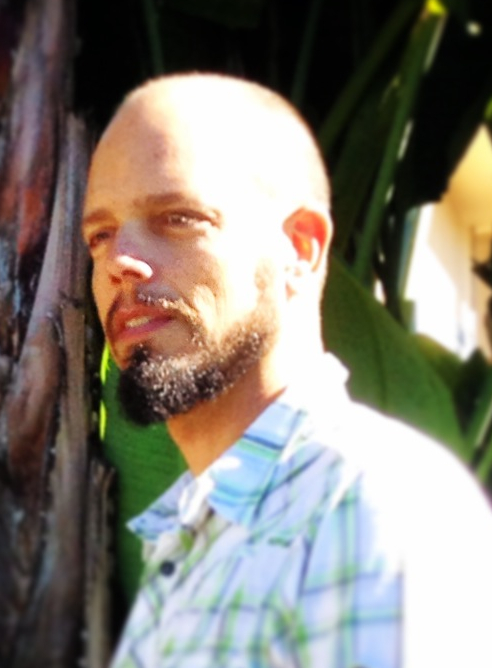}}]
  {Massimo Franceschetti} (M'98-SM'11) received the Laurea degree
  (with highest honors) in computer engineering from the University of
  Naples, Naples, Italy, in 1997, the M.S. and Ph.D. degrees in
  electrical engineering from the California Institute of Technology,
  Pasadena, CA, in 1999, and 2003, respectively.  He is Professor of
  Electrical and Computer Engineering at the University of California
  at San Diego (UCSD). Before joining UCSD, he was a postdoctoral
  scholar at the University of California at Berkeley for two
  years. He has held visiting positions at the Vrije Universiteit
  Amsterdam, the \'{E}cole Polytechnique F\'{e}d\'{e}rale de Lausanne,
  and the University of Trento. His research interests are in physical
  and information-based foundations of communication and control
  systems. He is co-author of the book ``Random Networks for
  Communication'' published by Cambridge University Press.
  Dr. Franceschetti served as Associate Editor for Communication
  Networks of the IEEE Transactions on Information Theory (2009 --
  2012), as associate editor of the IEEE Transactions on Control of
  Network Systems (2013-16) and as Guest Associate Editor of the IEEE
  Journal on Selected Areas in Communications (2008, 2009). He is
  currently serving as Associate Editor of the IEEE Transactions on
  Network Science and Engineering. He was awarded the C. H. Wilts
  Prize in 2003 for best doctoral thesis in electrical engineering at
  Caltech; the S.A.  Schelkunoff Award in 2005 for best paper in the
  IEEE Transactions on Antennas and Propagation, a National Science
  Foundation (NSF) CAREER award in 2006, an Office of Naval Research
  (ONR) Young Investigator Award in 2007, the IEEE Communications
  Society Best Tutorial Paper Award in 2010, and the IEEE Control
  theory society Ruberti young researcher award in 2012.
\end{IEEEbiography}

\begin{IEEEbiography}
  [{\includegraphics[width=1in,height=1.25in,clip,keepaspectratio]{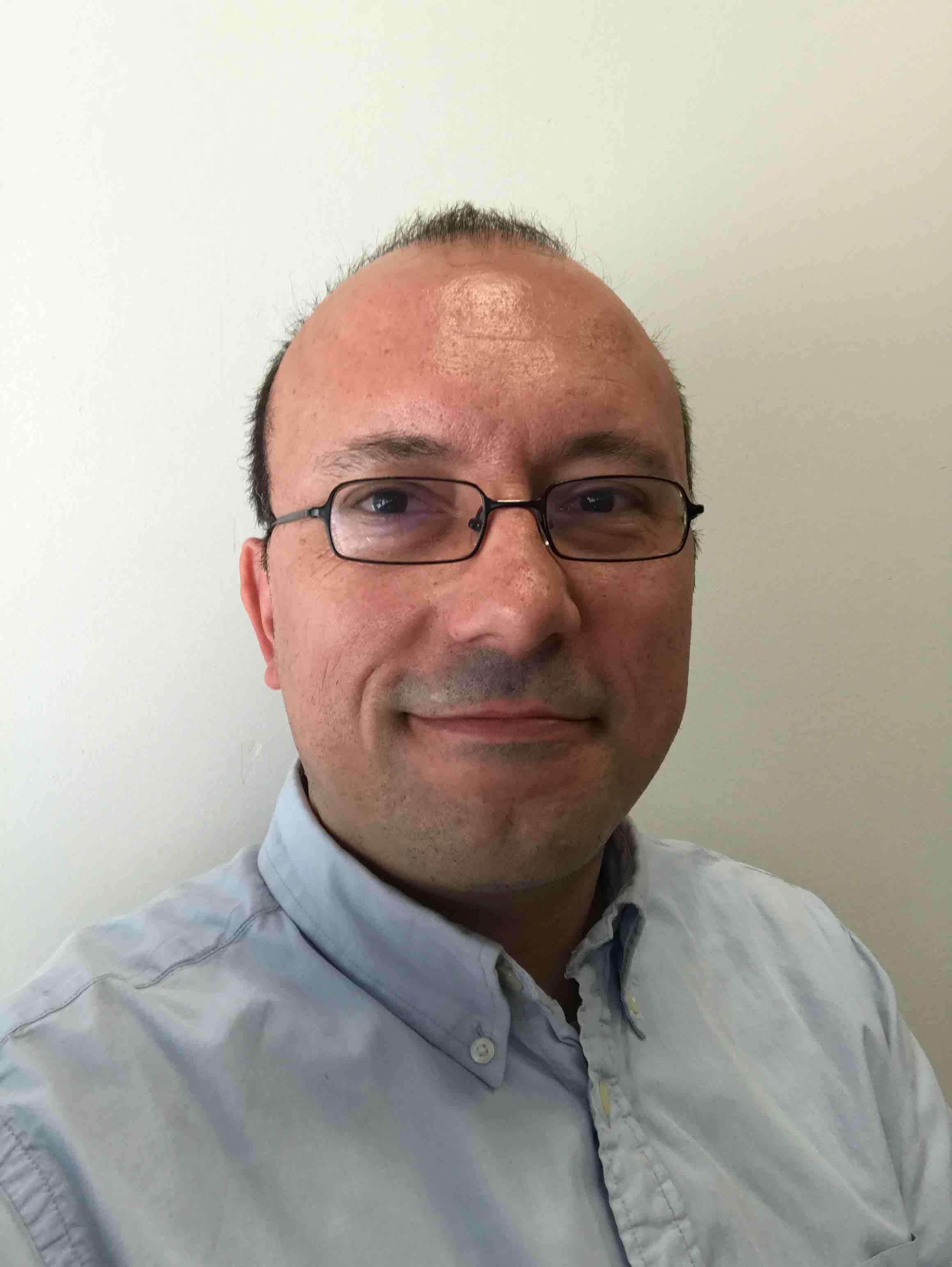}}]
  {Jorge Cort\'es} (M'02-SM'06-F'14) received the Licenciatura degree
  in mathematics from Universidad de Zaragoza, Zaragoza, Spain, in
  1997, and the Ph.D. degree in engineering mathematics from
  Universidad Carlos III de Madrid, Madrid, Spain, in 2001. He held
  postdoctoral positions with the University of Twente, Twente, The
  Netherlands, and the University of Illinois at Urbana-Champaign,
  Urbana, IL, USA. He was an Assistant Professor with the Department
  of Applied Mathematics and Statistics, University of California,
  Santa Cruz, CA, USA, from 2004 to 2007. He is currently a Professor
  in the Department of Mechanical and Aerospace Engineering,
  University of California, San Diego, CA, USA. He is the author of
  Geometric, Control and Numerical Aspects of Nonholonomic Systems
  (Springer-Verlag, 2002) and co-author (together with F. Bullo and
  S. Mart{\'\i}nez) of Distributed Control of Robotic Networks
  (Princeton University Press, 2009). He has been an IEEE Control
  Systems Society Distinguished Lecturer (2010-2014) and is an elected
  member for 2018-2020 of the Board of Governors of the IEEE Control
  Systems Society.  His current research interests include distributed
  control and optimization, network science, opportunistic
  state-triggered control and coordination, reasoning under
  uncertainty, and distributed decision making in power networks,
  robotics, and transportation.
\end{IEEEbiography}

\end{document}